\documentclass[11pt,a4paper]{article}    
\usepackage{amsmath,amsfonts,amssymb,amsthm,amscd}
\usepackage[english]{babel}

\newtheorem{thm}{Theorem}
\newtheorem{corollaire}{Corollary}
\newtheorem{prop}{Proposition}
\newtheorem{lemma}{Lemma}
\newtheorem{remark}{Remark}

\newcommand{\dsp}{\displaystyle}
\newcommand{\eps}{\varepsilon}
\newcommand{\R}{\mathbb{R}}

\newcommand{\RR}{\mathbb{R}^2}

\newcommand{\supp}{\mathrm{supp}}

\numberwithin{equation}{section}

\date{}                                          

\begin{document}
 
\title{On the attractive plasma-charge system in 2-d}

 \author{S. Caprino$\,^1$, C. Marchioro$\,^2$, E. Miot$\,^3$ and M. Pulvirenti$\,^4$}

\footnotetext 
[1]{Dipartimento di Matematica, Universit\`a di Tor Vergata,
Via della Ricerca Scientifica 1, 00133 Roma, Italy. E-mail: {caprino@mat.uniroma2.it}}
\footnotetext 
[2]{Dipartimento di
Matematica Guido Castelnuovo, Universit\`a La Sapienza, P.le A. Moro, 00185 Roma, Italy. E-mail:
 {marchior@mat.uniroma1.it}}
\footnotetext
[3]{Laboratoire de Math\'ematiques, Universit\'e Paris-Sud 11, B\^at. 425, 91405 Orsay,
France. E-mail: {evelyne.miot@math.u-psud.fr}}
\footnotetext 
[4]{Dipartimento di
Matematica Guido Castelnuovo, Universit\`a La Sapienza, P.le A. Moro, 00185 Roma, Italy. E-mail:
 {pulviren@mat.uniroma1.it}}

\maketitle

\begin{abstract}
 We study a positively charged Vlasov-Poisson plasma in which 
$N$ negative point charges are immersed. 
The attractiveness of the system forces us to consider a possibly unbounded plasma density
 near the charges. 
We prove the existence of a  global in time solution, assuming a suitable initial distribution of the velocities of the plasma particles. Uniqueness remains unsolved.
 \end{abstract}

\medskip

\medskip



\section{Introduction}

An interesting physical situation in Plasma Physics is when a system of $N$ heavy charged particles (say positive ions for instance) evolves in a plasma, that is a sea of light particles, of opposite sign (say electrons). The latter subsystem is often conveniently described in terms of a mean-field approximation by a continuous distribution $f(x,v,t)$, being $x$, $v$, $t$ position, velocity of a light particle and time respectively. Thus the time evolution of the full system is given by the following Vlasov-Poisson equation
\begin{equation}
(\partial_t +v\cdot \nabla_x+(E-F)\cdot \nabla_v)f(x,v,t)=0,
\label{VP0}
\end{equation}
where $E$ is the self-consistent electric field generated by the continuous 
charge distribution $f$ (see details below) and $F$ is the electric field generated by 
the point charges whose positions at time $t$ are denoted by $\xi^1(t), \ldots, \xi^N(t),$ that is
\begin{equation}
F(x,t)=\sum_{i=1}^N F^{i}( x,t),
\label{VP00}
\end{equation}with
\begin{equation}
F^{i}( x,t)=
\frac {x-\xi^i(t)}{|x-\xi^i(t)|^d} \ .
\label{VP1}
\end{equation}
Here $d=2,3$ denotes the dimension of the physical space and we are assuming, 
for notational simplicity, that charges and masses of the point charges are identical and unitary.

Equation \eqref{VP0} has to be complemented by the 
ODE describing the motion of the point charges which is, for any 
$i=1,\ldots,N,$ 
\begin{equation}
\dot \eta^i =-E(\xi^i,t)+ \sum_ {j\neq i=1}^N F^{j}( \xi^i,t),
\label{VP2}
\end{equation}
where $\eta^i =\dot \xi^i $ denotes the velocity of the $i$-th charge.

The purpose of this paper is to study the existence of a global solution to eq.ns \eqref{VP0}-\eqref{VP2}.
In the completely repulsive case, the issue of global existence and uniqueness of the solution
has already been approached and solved: 
in \cite{1} and \cite{CM1} for bounded and unbounded two-dimensional plasma distributions and in \cite{MMP} for a 
bounded plasma in three dimensions. 
 The background on which these papers and the present one are 
based is the complete and satisfactory theory for the usual Vlasov-Poisson 
equation (namely equation \eqref{VP0} with $F=0$) which has been developed in 
many articles as \cite{G,Ho1,Ho2,LP,Lo,Uk,P,S,W80,W93} and others. Of course adjoining point charges 
to a continuous charge distribution implies adding  singular forces as the field $F$ 
defined in \eqref{VP00}-\eqref{VP1}. 
Hence the existing theory is strongly perturbed and has to be deeply modified. Some papers related to this context 
are \cite{MM1,MM2,Sal,ZM}.

The essential tool employed in \cite{1} 
and \cite{MMP}, which will be used also here, is the study of a function $h(x,v,t)$ defined, in case of a single charge, as follows:
$$
 h(x,v,t)=\frac{|v-\eta(t)|^2}{2}+\sigma\ G(|x-\xi(t)|), 
 $$
where $G$ is the fundamental solution of the Poisson equation and $\sigma =\pm 1$ for the repulsive and attractive case 
respectively. It represents the energy of a plasma particle in the reference frame relative to the point charge, and of 
course  it is not a time invariant function. Nevertheless it  presents two essential properties: first, its time derivative 
does not depend upon the singular forces, which is crucial in proving its boundedness, and secondly, in the completely 
repulsive case ($\sigma=+1$), it has good sign properties to give a control on the velocity of the plasma particle and 
its distance from the charge. In the model discussed in this paper, that is in the attractive case,  
the second property is not satisfied, being  $\sigma =- 1.$ Consequently the sign of $h$ is not defined 
and hence getting a bound on $h$ does not imply that the single terms appearing in its definition are bounded. Indeed, 
in contrast with the repulsive case, here there can be plasma particles that, 
even starting far apart from the charge,  arrive close to it in a finite time while
gaining arbitrarily large velocities, but still having bounded energy. 

We stress that we succeed here in proving the global existence of the time evolution 
of system \eqref{VP0}-\eqref{VP2}, but not its uniqueness, for which we think that new considerations are needed.

The paper consists of six sections. After the introduction, in Section \ref{sec:one-charge} we pose the problem and 
present the main result (Theorem \ref{thm:main}), stated for a system consisting of a positive plasma density and a single, negative point 
charge. We start with the single-charge case in order to provide a proof which is clear and 
contains all essential tools that are needed also for the $N$-charges case. 
Sections \ref{sec:approx} to \ref{section:compactness} are then devoted to the proof of Theorem \ref{thm:main};
In Section \ref{sec:approx} we introduce a family of regularized differential systems, for which we 
establish many preliminary estimates holding uniformly with respect to the regularization. In Section \ref{sec:energy} 
we show the main technical result in this paper, 
that is the boundedness of the above mentioned function $h$ (see Theorem \ref{thm:main2}). 
This allows, in Section \ref{section:compactness}, 
to prove the existence of a global solution of the system, obtained as limit of the regularized dynamics. 
Finally in Section \ref{N} we state and prove global existence of a solution to system \eqref{VP0}-\eqref{VP2} 
for $N$ charges (Theorem \ref{thm:mainN}).
 
 \section{The result for a single charge}\label{sec:one-charge}

In this section we consider a plasma in two dimensions with only one charge. 
We set $( \xi(t),\eta(t))$ for position and velocity of the charge at 
time $t$, being $( \xi, \eta)$ their initial data.
Moreover $f=f(t)\in L^\infty(L^1\cap L^\infty)$ denotes the density of the plasma, and we assume that $f(0)=f_0$
is a bounded probability density.

Equation \eqref{VP0} describes a conservation law for the density 
along the time evolution of the characteristics, which is, at least formally, 
given by  the following differential system:
\begin{equation}
\label{eq:system}
\begin{cases}
\dsp  \dot{x}(x,v,t)=v(x,v,t) \\\vspace*{0.5em}
\dsp  \dot{v}(x,v,t)= (E-F)\left(x(x,v,t),t\right)\\\vspace*{0.5em}
\dsp ( x(x,v,0), v(x,v,0))=(x,v)\in \R^2\setminus\{\xi\}\times \R^2 \\\vspace*{0.5em}
\dsp E(x,t)=\int_{\R^2} \frac{x-y}{|x-y|^2}\rho(y,t)\,dy\\\vspace*{0.5em}
\dsp \rho(x,t)=\int_{\R^2} f(x,v,t)\,dv\\\vspace*{0.5em}
\dsp F(x,t)=\frac{x-\xi(t)}{|x-\xi(t)|^2}\\\vspace*{0.5em}
 \dsp f\left(x(x,v,t),v(x,v,t),t\right)=f_0(x,v),
 \end{cases}
\end{equation}
together with the evolution of the charge, moving according to:
\begin{equation}
\label{ch}
\begin{cases}
\dsp \dot{\xi}(t)=\eta(t)\\
\dsp \dot{\eta}(t)=-E(\xi(t),t)\\
\dsp (\xi(0), \eta(0))=(\xi,\eta)\in \R^2\times \R^2.
 \end{cases}
 \end{equation}
Note that, if $f_0$ is smooth, then any solution $(\xi(t),\eta(t);f(t))$ to system \eqref{eq:system}-\eqref{ch} 
satisfies the Vlasov-Poisson equation \eqref{VP0}-\eqref{VP2}. 
Clearly, the ODE  in \eqref{eq:system}  are not well-defined if some plasma particles  
collide with the charge in finite time;
However Theorem \ref{thm:main} below ensures that, despite the attractive interaction between plasma and
charge, such collapses can be essentially avoided under
suitable assumptions on the support of $f_0$.

\medskip

We introduce the function:
\begin{equation*}
 h(x,v,t)=\frac{1}{2}|v-\eta(t)|^2+\ln|x-\xi(t)| 
\end{equation*}
which represents the energy of a plasma particle in the reference frame of the moving charge. 
Moreover we set $S_0$ for the support of  $f_0$, which can be possibly an unbounded set, 
and we define the quantity
\begin{equation*}
\mathcal{H}(t)=\sup_{s\in[0,t]}\sup_{(x,v)\in S_0}|h(x(t),v(t),t)|+C 
\end{equation*}
with $C$ a sufficiently large constant for further purposes. 
We will prove that if $\mathcal{H}(0)$ is finite, then $\mathcal{H}(t)$ remains finite on bounded time intervals and consequently the velocities of the plasma particles are logarithmically diverging as they approach 
the charge. Nevertheless such slight divergence will not prevent us 
to prove  global existence of a solution to 
\eqref{eq:system}-\eqref{ch}.   

\medskip

In the sequel we will often use the notation
\begin{equation*}
\ln_-r= -\ln r\ \chi( r\in (0,1] ), 
\end{equation*}
with $\chi(A)$ the characteristic function of the set $A.$ 

We will set $C$ for a positive constant and $C_i,$ $i=1,2,\ldots,$ 
for some constants to be quoted in the course of the paper. All of 
them will possibly depend on $||f_0||_{L^{\infty}},$ $||f_0||_{L^{1}}$ and on an arbitrarily fixed time $T.$
Finally, for sake of brevity  we will sometimes use the 
shortened notation $( x(t),v(t))$ instead of $(x(x,v,t),v(x,v,t))$.

\medskip

Our main result is the following:

\begin{thm}\label{thm:main}
 Let $(\xi,\eta)\in \R^2\times \R^2$ and $f_0\in L^\infty(\R^2\times \R^2)$ be a probability density supported on the set 
 \begin{equation}
S_0=\left\{(x,v)\in \R^2\times \R^2:   |h(x,v,0)|\leq C_0\right\}\label{hp}
\end{equation}
for some positive $C_0.$

Let $T>0$. Then there exist 
\begin{equation*}
 \begin{split}
&f\in L^\infty\left([0,T]; L^\infty\cap L^1(\R^2\times \R^2)\right),\quad 
E\in L^\infty\left([0,T];L^\infty(\R^2)\right),\\
& (\xi(\cdot),\eta(\cdot))\in C^1([0,T])^2,\end{split}\end{equation*}
and for $d\mu_0$-a.a. $(x,v)\in S_0$ there exists $(x(\cdot),v(\cdot))\in C^1([0,T])^2$
 such that  $\big(x(t),v(t);\xi(t),\eta(t);f(t)\big)$
satisfy system  \eqref{eq:system}-\eqref{ch} on $[0,T]$. In particular,
 for $d\mu_0$-a.a. $(x,v)\in S_0$ and $\forall t\in [0,T]$ we have $|x(t)-\xi(t)|>0$  and
 \begin{equation}
 f(x(t),v(t),t)=f_0(x,v).\label{inv}
\end{equation}

Moreover,
\begin{equation}
\mathcal{H}(T)\leq C
 \label{p10}
\end{equation}
and finally
\begin{equation}\label{p1}
 |\rho(x,t)|\leq C\left(1+\ln_-|x-\xi(t)|\right),\quad (x,t)\in \RR\times [0,T].
\end{equation}

\end{thm}
\begin{remark}\label{rem:1}
 In the proof of Theorem \ref{thm:main} we shall use the fact that \eqref{p1}, which is a consequence of \eqref{p10},
 ensures the uniqueness of the solutions
$\big(x(t),v(t);\xi(t),\eta(t)\big)$  to the ODE in \eqref{eq:system}-\eqref{ch} once $f(t)$ (hence $E(t)$)
 is given (see Corollary
\ref{cor:1} and Lemma \ref{lemma:gronwall}). However notice that we are not claiming the uniqueness of the triple $\big( x(t),v(t);\xi(t),\eta(t);f(t)\big).$

\end{remark}

\begin{remark}\label{rem:2}

We stress a significant difference with respect to the repulsive case treated in \cite{1}  and \cite{MMP}.  
In those papers the assumption for $\mathcal{H}(0)$ to be finite was equivalent to assuming a 
finite distance between charge and plasma at time $t=0,$ while in the present case it is not so. This is an intrinsic difficulty in this setup since, even assuming an initial positive distance between plasma and charge, we could not exclude that some plasma particle arrive at any prefixed distance from the charge. 
\end{remark}

\begin{remark}\label{rem:3}
Hypothesis \eqref{hp} on the support of $f_0$ ensures that its spatial support is bounded. 
Indeed we have 
$|x-\xi|\leq  e^{C_0-\frac{1}{2}|v-\eta|^2}\leq e^{C_0}.$ By \eqref{p10} this property will be preserved in time (see \eqref{supp}).
\end{remark}

\begin{remark}\label{rem:4}
We do not claim that the bound \eqref{p1} is optimal.
\end{remark}

\section{The approximating system}
\label{sec:approx}
In this section, we introduce a regular version of the original system \eqref{eq:system}-\eqref{ch} 
 by mollifying the singular field $F$ created by the  charge. More 
precisely, for a small parameter $0<\eps<1$, we consider the smooth increasing function  $ \ln_\eps : [0,+\infty)\to \R$ 
such that
\begin{equation}
 \ln_\eps r\geq 2\ln \eps\quad \text{if}\quad r\leq \eps,\quad \quad \ln_\eps r=\ln r\quad 
\text{if}\quad r\geq \eps. \label{logg}
\end{equation}

We consider next the unique solution 
$\big((x_\eps(t),v_\eps (t); \xi_\eps(t),\eta_\eps(t); f_\eps(t)\big)$ to the following $\eps$-problem on $[0,T]:$ 
\begin{equation}
\label{eq:syst1}
\begin{cases}
  \dsp   \dot{v}_\eps=(E_\eps-F_\eps)(x_\eps,t) \\ 
\dsp (x_\eps(0),v_\eps(0))=(x, v)\\
\dsp   \dot{\eta}_\eps=-E_\eps(\xi_\eps,t)\\
\dsp ( \xi_\eps(0),\eta_\eps(0))=(\xi,\eta)
 \end{cases}
\end{equation}
where:
\begin{equation}\begin{cases}
\dsp E_\eps(x,t)=\int \rho_\eps(y,t)\frac{x-y}{|x-y|^2}\,dy\\
\dsp F_\eps(x,t)=\nabla_x \ln_\eps |x-\xi_\eps(t)|\\
\dsp \rho_\eps(y,t)=\int f_\eps(y,v,t)\,dv
\end{cases}\end{equation}
and satisfying:
\begin{equation}
 \label{eq:transport-one}
 f_\eps\left(x_\eps(t),v_\eps(t),t\right)=f_0^\eps(x,v),
\end{equation}
where $f_0^\eps$ is a smooth, compactly supported approximation of $f_0$.
Here we are in presence of the Vlasov-Poisson problem with an additional smooth gradient external 
field, for which the classical theory for global existence and uniqueness of the solution applies with minor modifications. 

Thanks to \eqref{eq:transport-one} we have
$
\|f_\eps(t)\|_{L^\infty} = \|f_0^\eps\|_{L^\infty}.
$
 Moreover, since \eqref{eq:syst1} is hamiltonian the flow $(x,v)\mapsto (x_\eps(x,v,t),v_\eps(x,v,t))$ 
preserves 
the Lebesgue's measure on $\R^2\times \R^2$ for all $t\geq 0$. Consequently all the $\|f_\eps(t)\|_{L^p}$ norms ($p\geq 1)$ are conserved.

We introduce next the regularized relative energy per plasma particle:
\begin{equation}
 h_\eps(x,v,t)=\frac{1}{2}|v-\eta_\eps(t)|^2+\ln_\eps|x-\xi_\eps(t)| \label{h}
\end{equation}
and we set:
\begin{equation}
\mathcal{H}_\eps(t)=\sup_{s\in[0,t]}\sup_{(x,v)\in S_0}|h_\eps(x_\eps(t),v_\eps(t),t)|+C_1.\label{h'}
\end{equation}
Notice at this point that definition \eqref{h'} does not 
allow us to consider initial data $f_0$ satisfying 
assumption \eqref{hp}, since there are configurations of particles for which $h(x,v,0)$ is 
bounded while $h_\eps(x,v,0)$ is not. 
This is due to the smoothed potential $\ln_\eps$ which cannot compensate large velocities of 
particles that are very close to the charge. To overcome this difficulty we introduce another positive  
parameter $\beta>\eps$ and, instead of considering initial data $f_0^\eps$, we consider $f^\beta_0$ supported in the set
\begin{equation}
 S_0^{\beta}=\left\{(x,v)\in \R^2\times \R^2:   |h(x,v,0)|\leq C_0, |x-\xi|> \beta \right\}
 \subset S_0. \label{beta}
\end{equation}
We observe that $S_0^\beta$ is a bounded set (see Remark \ref{rem:3}). Setting now
\begin{equation}
\mathcal{H}_\eps^\beta(t)=\sup_{s\in[0,t]}\sup_{(x,v)\in S_0^\beta}|h_\eps(x_\eps(t),v_\eps(t),t)|+C_1\label{h''}
\end{equation}
we have by definition \eqref{beta}:
\begin{equation}
\mathcal{H}_\eps^\beta(0)=\mathcal{H}(0)<C_0\label{acca0}
\end{equation}
and $f_0^\beta$ is a compactly supported function satisfying the assumptions of Theorem \ref{thm:main}.

We will prove that the solution $f_\eps^{\beta}(t)$ to system \eqref{eq:syst1} with initial condition $f_0^{\beta}$  
enjoys estimates independent of $\eps$ and $\beta$ making it possible to pass to the limit as $\eps\to 0$ and $\beta \to 0.$
To simplify the notation from now on we will sometimes omit the index $\beta,$ but we keep in mind that the solution depends 
on both parameters.  We emphasize that all constants appearing in what follows do not depend on $\eps$ and $\beta.$

Let us introduce the total energy of the system, which is a conserved quantity:
\begin{equation}
\begin{split}
& \mathcal{E}_\eps(0)= 
\mathcal{E}_\eps(t)=\frac12\int |v|^2f_\eps(x,v,t)\,dx\,dv+\frac{|\eta_\eps(t)|^2}{2}+C_2\\ &-\frac12\iint \ln|x-y|\rho_\eps(x,t)\rho_\eps(y,t)\,dx\,dy
+\int \ln_\eps |x-\xi_\eps(t)|\rho_\eps(x,t)\,dx.
\end{split}\label{energy}
\end{equation}

Our first observation is that $ \mathcal{E}_\eps(0)$ is positive and bounded uniformly in $\eps$ and $\beta$, 
as it is stated in the following 
\begin{prop}
\label{prop:E0}
If $f_0$ is supported on the set $S_0^\beta$ given by \eqref{beta}, then every 
single term appearing in the definition of $\mathcal{E}_\eps(0)$ is bounded. As a consequence we have:
$$
0\leq \mathcal{E}_\eps(0)\leq C_3
 $$
 provided that $C_2$ has been chosen large enough.\end{prop}
 \begin{proof}
 We start by recalling the following elementary fact: $\forall p> 0$ there exists a positive constant $C(p)$ such that
 \begin{equation}
 \int(\ln_-r)^pdr=C(p).  \label{elem}
 \end{equation}

Next, for any $(x,v)\in S_0^{\beta}$ we infer from definitions \eqref{h} and  \eqref{logg} that
\begin{equation*}
\begin{split}
| v|^2&\leq 2( | v- \eta|^2+| \eta|^2)\\&\leq 4\big(h_\eps(x,v,0)+\ln_-| x-\xi|\big)+2|\eta|^2.
\end{split}
\label{v0}\end{equation*}
Therefore  we deduce from definitions \eqref{h''}-\eqref{acca0} that, provided the constant $C_1$ is sufficiently large,
\begin{equation}| v|^2\leq C(\mathcal{H}(0)+\ln_-| x- \xi|)\quad \forall (x,v)\in S_0^\beta. \label{v01}
\end{equation}

This implies that for $x\in \RR$ we have
\begin{equation*}
\begin{split}
\rho_0(x)&=\int_{\{v:(x,v)\in S_0^\beta\}}\ f_0(x,v)\,dv\\&\leq 
||f_0||_{\infty}\int  \ \chi\Big(| v|\leq
 C\sqrt{ \mathcal{H}(0)+\ln_-| x- \xi|}\Big)\,dv,
\end{split}\end{equation*}
hence
\begin{equation}
\rho_0(x)\leq C\big(\mathcal{H}(0)+\ln_-| x- \xi|\big),\quad \forall x\in \RR. 
\label{dens1}
\end{equation}

Estimates \eqref{v01} and \eqref{dens1} enable us to prove that any single term
 appearing in the definition \eqref{energy} of $ \mathcal{E}(0)$ is bounded. Indeed, 
by Remark \ref{rem:3} on the compactness of the support of $\rho_0$ and by \eqref{dens1} and \eqref{elem} we have 
\begin{equation}
\int |\ln_\eps |x-\xi||\rho_0(x)\,dx\leq C.\label{P0}
\end{equation}
 Moreover, for the same reason:
\begin{equation*}
\begin{split}
\int |v|^2f_0(x,v)\,dx\,dv&\leq C\int (\ln_-| x- \xi|+1)f_0(x,v)\, dx\,dv \\ 
&\leq C\left(1+\int \ln_-| x- \xi|\left(\ln_-| x- \xi|+1\right)\ dx\right)
\end{split}\end{equation*}
hence
\begin{equation*}
\int |v|^2f_0(x,v)\,dx\,dv\leq C.
\end{equation*}
Finally:
\begin{equation*}
\begin{split}
\iint \big| &\ln|x-y|\big|\rho_0(x)\rho_0(y)\,dx\,dy\\\leq 
&C\iint_{\supp \rho_0}|\ln|x-y||(\ln_-| x- \xi|+1)(\ln_-| y- \xi|+1)  \,dx\,dy.\end{split}
\end{equation*}
Again by \eqref{elem} the above integral can be easily bounded by means of 
Cauchy-Schwarz inequality, so that   $ \mathcal{E}_\eps(0)$ is bounded uniformly in $\eps$ and positive, provided $C_2$ is sufficiently large.
\end{proof}
The preceding result does not give us any $\eps$-uniform bound 
on the single terms composing $ \mathcal{E}_\eps(t)$, since it could be bounded uniformly in $\eps$ by compensation. 
The next two results provide such informations. Their proof is extensively based on 
the conservation of the Lebesgue's measure and on the invariance of the plasma density  along the motion of the characteristics.

We set
$$
K_\eps(t)=\frac12\int |v|^2f_\eps(x,v,t)\,dx\,dv+\frac{|\eta_\eps(t)|^2}{2}.
$$
\begin{prop}
\label{prop:unif}
 \begin{equation}
 \sup_{t\in [0,T]}K_\eps(t)\leq C_4,\label{kin}
 \end{equation}
\begin{equation}
\sup_{t\in[0,T]}\int |x|\rho_\eps(x,t)\,dx\leq C_5,\label{mom}
\end{equation}
\begin{equation}
\sup_{t\in[0,T]}\|\rho_\eps(t)\|_{L^2}\leq C_6.\label{norm2}
\end{equation}

\end{prop}

\begin{proof}
 For any $M\geq 0$, we have
\begin{equation*}
 \begin{split}
  \rho_\eps(x,t)&=\int_{|v|< M} f_\eps(x,v,t)\,dv+\int_{|v|\geq M}f_\eps(x,v,t)\,dv\\
&\leq  \pi M^2 \|f_\eps(t)\|_{L^\infty}+\frac{1}{M^2}\int |v|^2 f_\eps(x,v,t)\,dv.
 \end{split}
\end{equation*}
 By optimizing in $M$ we find:
\begin{equation*}
 \rho_\eps(x,t)\leq C\left( \int |v|^2 f_\eps(x,v,t)\,dv\right)^{1/2},
\end{equation*}
 whence, by definition of $K_\eps(t)$:
\begin{equation}
 \label{ineq:rho}
\|\rho_\eps(t)\|_{L^{2}}\leq C\sqrt{K_\eps(t)}.
\end{equation}

On the other side, from the energy conservation and Proposition \ref{prop:E0} it follows that
\begin{equation}
\label{pp}\begin{split}
 K_\eps(t)&\leq C _3+\int\ \ln_-|x-\xi_\eps(t)|\rho_\eps(x,t)\,dx\\&+
 \frac12\iint_{|x-y|\geq 1} 
\ln|x-y|\rho_\eps(x,t)\rho_\eps(y,t)\,dx\,dy.
\end{split}
\end{equation}
Now,  Cauchy-Schwarz inequality and \eqref{ineq:rho} yield
\begin{equation}
\sup_z\int\ \ln_-|x-z|\rho_\eps(x,t)\,dx\leq C\|\rho_\eps(t)\|_{L^2}\leq C\sqrt{K_\eps(t)}. \label{rho^2}
\end{equation}
Moreover    $0\leq\ln r\leq r$ for any $ r\geq 1$. 
Hence \eqref{pp} and \eqref{rho^2} imply:
\begin{equation}
\begin{split} K_\eps(t)&\leq C_3+C\sqrt{K_\eps(t)}+\frac12\iint (|x|+|y|)\rho_\eps(x,t)\rho_\eps(y,t)\,dx\,dy\\
&\leq C_3+C\sqrt{K_\eps(t)}+\int |x|\rho_\eps(x,t)\,dx.
\label{ineq:rho2}\end{split}
\end{equation}

Next, by \eqref{eq:transport-one} and the fact that the flow preserves the Lebesgue's measure on $\RR\times \RR$ we have
\begin{equation}
 \begin{split}
\int |x|\rho_\eps(x,t)\,dx
&=\int |x_\eps (x,v,t)|f_0(x,v)\,dx\,dv\\
&\leq  \int \left(|x|+\int_0^t |v_\eps(x,v,s)|\,ds\right)f_0(x,v)\,dx\,dv.\end{split}
\end{equation}
Recalling Remark \ref{rem:3}, 
$\rho_0$ has compact support, so that applying again Cauchy-Schwarz inequality we get
\begin{equation}
 \begin{split}
\int |x|\rho_\eps(x,t)\,dx&\leq  C+\int_0^t \int |v|f_\eps(x,v,s)\,dx\,dv\,ds\\ 
&\leq C\left(1+\int_0^t \sqrt{K_\eps(t)}\,ds\right).
\end{split}
\label{ineq:rho3}\end{equation}

Going back to \eqref{ineq:rho2}, \eqref{ineq:rho3} implies: 
\begin{equation*}
\begin{split}
K_\eps(t)\leq C+C\sqrt{K_\eps(t)}+C \int_0^t \sqrt{K_\eps(t)}\,ds
\end{split}
\end{equation*}
and \eqref{kin} follows then from Gronwall's Lemma. Finally recalling \eqref{ineq:rho3} and \eqref{ineq:rho}
 we conclude that \eqref{mom} and \eqref{norm2} follow from \eqref{kin}.
\end{proof}

\begin{remark}\label{rem:velocity-charge}
We stress the fact that \eqref{kin} yields 
a bound on the velocity of the charge and consequently on  its motion, which remains confined over the interval $[0,T].$
\end{remark}
Proposition \ref{prop:unif} implies the following bounds on the potential terms in the energy:
\begin{corollaire}
\label{pot}
\begin{equation*}\begin{split}
&\sup_{t\in [0,T]}\int|\ln_\eps |x-\xi_\eps(t)||\rho_\eps(x,t)\,dx
\leq C\\
&\sup_{t\in [0,T]}\iint |\ln |x-y||\rho_\eps(x,t)\rho_\eps(y,t)\,dx\,dy\leq C.
\end{split}
\end{equation*}
\end{corollaire}

\begin{proof}
Arguing as in  \eqref{pp}-\eqref{ineq:rho3}, by Proposition \ref{prop:unif} we have thanks to \eqref{rho^2} and 
\eqref{ineq:rho3}   
\begin{equation}
\label{est:energy1}
\begin{split}
-\int_{|x-\xi_\eps(t)|\leq 1}&\ln_\eps |x-\xi_\eps(t)|\rho_\eps(x,t)\,dx\\&+\frac12\iint_{|x-y|\geq 1}
 \ln |x-y|\rho_\eps(x,t)\rho_\eps(y,t)\,dx\,dy\leq C.
 \end{split}
\end{equation}
On the other side from definition \eqref{energy} of the energy it follows that:
\begin{equation}
 \label{est:energy2}
\begin{split}
\int_{|x-\xi_\eps(t)|\geq 1}&\ln_\eps |x-\xi_\eps(t)|\rho_\eps(x,t)\,dx+\frac12\iint
 \ln_- |x-y|\rho_\eps(x,t)\rho_\eps(y,t)\,dx\,dy\\
& \leq  C_3-\int_{|x-\xi_\eps(t)|\leq 1}\ln_\eps |x-\xi_\eps(t)|\rho_\eps(x,t)\,dx\\
&+\frac12\iint_{|x-y|\geq 1}
 \ln |x-y|\rho_\eps(x,t)\rho_\eps(y,t)\,dx\,dy.
 \end{split}
\end{equation}
 Hence the conclusion follows from \eqref{est:energy1} and \eqref{est:energy2}.\end{proof}

\section{The function $H_\eps$}
\label{sec:energy}

The main result of this section is the following

\begin{thm}
\label{thm:main2}
\begin{equation*}
 \mathcal{H}_\eps(T)\leq C.
\end{equation*}
\end{thm}

The proof of Theorem \ref{thm:main2} requires some preliminary results, which are stated hereafter.
\begin{lemma} \label{logar} For all $(x,v)\in S_0^{\beta}$ and for any $t\in [0,T]$ it holds:
\begin{equation}
\label{u}
| v_\eps(t)|\leq C\sqrt{ {\cal{H}}_{\eps}(t)+\ln_-| x_\eps(t)- \xi_\eps(t)|},
\end{equation}
\begin{equation}
\rho_\eps(x,t)\leq C\big({\cal{H}}_{\eps}(t)+\ln_-| x- \xi_\eps(t)|\big),
\label{uu}
\end{equation}
and 
\begin{equation}
|x_\eps(t)|\leq C\sqrt{{\cal{H}}_{\eps}(t)}.
\label{suppro}
\end{equation}
\end{lemma}
\begin{proof}
We infer from Remark \ref{rem:velocity-charge} that 
\eqref{u} and \eqref{uu} are the equivalent to \eqref{v01} and \eqref{dens1} at time $t.$

Now we prove \eqref{suppro}. Recalling Remark \ref{rem:3} in Section \ref{sec:one-charge}, let $B(0,R)$ be the ball 
of radius $R>1$, so large that it contains the support of $\rho_0$ and so that moreover $\xi_\eps(t)\in B(0,R), \forall 
t\in [0,T].$ Let us fix $t\in [0,T]$ and define $t^*=\max \{s\in [0,t]: x_\eps(s)\in B(0,2R)\}.$ Then either 
$t^*=t$ or $t^*<t$ and in the latter case we have  $|x_\eps(s)-\xi_\eps(s)|>R>1$ for any $s\in (t^*,t].$ 
Hence in view of the definition \eqref{h} of $h_\eps$ it follows that for $s\in (t^*,t]$ we have $|v_\eps(s)|\leq \sqrt{2\mathcal{H}_\eps(s)},$ which implies 
$$
|x_\eps(s)|\leq 2R+\int_{t^*}^{s}|v_\eps(\tau)|d\tau \leq C\sqrt{\mathcal{H}_\eps(s)}
$$
 provided the constant $C_1$ in \eqref{h'} has been chosen sufficiently large.
\end{proof}

\begin{prop}
\label{coro:E} For $t\in [0,T]$ we have
 \begin{equation*}
 \|E_\eps(t)\|_{L^\infty}\leq C \sqrt{\ln\mathcal{H}_\eps(t)}.
\end{equation*}
\end{prop}

\begin{proof}
We decompose $E_\eps(x,t)$ as
\begin{equation*}
\begin{split}
 E_\eps( x,t) 
 ={\cal{I}}_1( x,t)+{\cal{I}}_2( x,t), \label{q33}
\end{split}
\end{equation*}
where 
\begin{equation*}
{\cal{I}}_1( x,t)=\int_{| x-y|\leq \delta}  \rho_\eps( y,t) \frac{x- y}{ | x- y|^2}\,dy 
\end{equation*}
\begin{equation*}
{\cal{I}}_2( x,t)=\int_{|x-y|>\delta } \rho_\eps( y,t) \frac{x- y}{ | x- y|^2}\,dy \end{equation*}
and  $0<\delta <1$ is to be determined hereafter. 
Let us first estimate the term ${\cal{I}}_1( x,t).$ We have by \eqref{uu}:
\begin{equation*}\begin{split}
|{\cal{I}}_1( x,t)|&\leq  
C\int_{|x-y|\leq \delta} \ \frac{{\cal{H}}_{\eps}(t)+\ln_-| y- \xi_{\eps}(t)|}{| x-y|}\,dy\\
&\leq C{\cal{H}}_{\eps}(t)\delta+ C\int_{|x-y|\leq \delta} \ \frac{\ln_-| y- \xi_{\eps}(t)|}{| x-y|}\,dy. \end{split}\label{q}
\end{equation*}
By considering the two cases:
$$
|y-\xi_\eps(t)|\leq| x- y| \ \ \ \hbox{and}\ \ \   |y-\xi_\eps(t)|>| x- y|
$$
we arrive at
\begin{equation}
\int_{|x-y|\leq \delta} \ \frac{\ln_-| y- \xi_{\eps}(t)|}{| x-y|}\,dy
\leq \int_{ |z|\leq\delta} \ \frac{|\ln|z||}{|z|}\,dz\leq C\delta |\ln \delta|,
\label{int-}\end{equation}
so that
\begin{equation}
\begin{split}
|{\cal{I}}_1( x,t)|\leq \ C{\cal{H}}_{\eps}(t)\delta+C\delta|\ln \delta|\leq C{\cal{H}}_{\eps}(t) \delta+C.
\end{split}\label{int}
\end{equation}
On the other side, applying Cauchy-Schwarz inequality and using \eqref{norm2} we get:
\begin{equation*}
\begin{split}
|{\cal{I}}_2( x,t)|&\leq  \int_{\delta<|x-y|\leq 1 } \ \frac{\rho_{\eps}(y,t)}{| x-y|}\,dy+
\int_{|x-y|> 1 } \  \frac{\rho_{\eps}(y,t)}{| x-y|}\,dy\\
&\leq 
\|\rho_\eps(t)\|_{L^2}\left( \int_{\delta<|z|\leq 1 }  \frac{dz}{| z|^2}\right)^{1/2}+1\leq C\sqrt{|\ln\delta|}.
\end{split}\label{int2}
\end{equation*}
With the choice
$$
\delta={\cal{H}}_\eps(t)^{-1}
$$
we obtain the thesis, provided the constant $C_1$ in the definition \eqref{h''} is large enough. 
\end{proof}

\begin{prop}\label{adjoined}For $t\in [0,T]$ we have
$$
\sup_{x\in \R^2} \int \ \  | w| \ \frac{f_{\eps}( y, w,t)}{ | x- y|}\,d y\,d w\leq C\sqrt{ {\cal{H}}_{\eps}(t)\ln  {\cal{H}}_{\eps}(t) }.
 $$
 \end{prop}
\begin{proof}
Let $0<\delta<1$ to be chosen hereafter. 
We perform the integral as follows:
$$
 \int \ \  | w| \ \frac{f_{\eps}( y, w,t)}{ | x- y|}\,d y\,d w={\cal{I}}_1( x,t)+{\cal{I}}_2( x,t)+{\cal{I}}_3( x,t),
 $$
 with
 \begin{equation*}\begin{split}
 {\cal{I}}_1( x,t)&= \int _{|x-y|\leq \delta} | w| \ \frac{f_{\eps}( y, w,t)}{ | x- y|}\,d y\,d w\\
 {\cal{I}}_2( x,t)&= \int _{\delta <|x-y|\leq 1}   | w| \ \frac{f_{\eps}( y, w,t)}{ | x- y|}\,d y\,d w\\
{\cal{I}}_3( x,t)&= \int_{|x-y|> 1}   | w| \ \frac{f_{\eps}( y, w,t)}{ | x- y|}\,d y\,d w.
\end{split}\end{equation*}
 Recalling estimates \eqref{u} and \eqref{uu} and proceeding   analogously to Proposition \ref{coro:E}, we get:
\begin{equation*}
\begin{split}
  {\cal{I}}_1( x,t)&\leq 
C\int_{|x- y|\leq \delta}  \frac{\big({\cal{H}}_{\eps}(t)+\ln_-| y-\xi_{\eps}(t)|\big)^{3/2}}{| x- y|}\,dy \\
&\leq C\delta {\cal{H}}_{\eps}(t)^{3/2}+ C\int_{|z|\leq  \delta} \ \frac{|\ln|z||^{3/2}}{| z|}\,dz\\
&\leq C\delta {\cal{H}}_{\eps}(t)^{3/2}+ C .
\end{split}\label{w1}
\end{equation*}
Next, using Cauchy-Schwarz inequality, estimate \eqref{kin} 
on the kinetic energy and \eqref{uu} we obtain:
\begin{equation*}
\begin{split}
 {\cal{I}}_2( x,t)&\leq \left(\int \   | w|^2 \ f_{\eps}(y, w,t)\,d y\,dw\right)^{1/2}
\left(\int_{\delta<|x- y|\leq 1} \  \ \frac{f_{\eps}(y, w,t)}{ | x- y|^2}\,d y\,dw\right)^{1/2}\\
&\leq 
 C\left(\int_{\delta<|x- y|\leq 1} \  \ \frac{\rho_{\eps}(y, t)}{ | x- y|^2}\,d y\right)^{1/2}\\
&\leq C\left(\int_{\delta<|x- y|\leq 1 } \ \frac{\big({\cal{H}}_{\eps}(t)+\ln_-| y-\xi_{\eps}(t)|\big)}{| x- y|^2}\,dy\right)^{1/2}\\
&\leq C\sqrt{{\cal{H}}_{\eps}(t)|\ln \delta|} +C\left(\int_{\delta<|x- y|\leq 1 } 
\ \frac{\ln_-| y-\xi_{\eps}(t)|}{| x- y|^2}\,dy\right)^{1/2}.
 \end{split}
 \end{equation*}
 Now, arguing as in in \eqref{int-} we have
 $$
\int_{\delta<|x- y|\leq 1 } dy\ \frac{\ln_-| y-\xi_{\eps}(t)|}{| x- y|^2}\leq C\left(\ln \delta\right)^2,
$$
hence
\begin{equation*}
 {\cal{I}}_2( x,t)\leq C\sqrt{{\cal{H}}_{\eps}(t)|\ln \delta|}+ C|\ln \delta|.
\end{equation*}

Finally, using again \eqref{kin} we arrive at
$$
 {\cal{I}}_3( x,t)\leq C.
 $$
 Hence we conclude that
 $$
  \int \ \  | w| \ \frac{f_{\eps}( y, w,t)}{ | x- y|}\,d y\,d w
\leq C\delta {\cal{H}}_{\eps}(t)^{3/2}+C\sqrt{{\cal{H}}_{\eps}(t)|\ln \delta|}+ C|\ln \delta| +C
  $$
  and the thesis is achieved by choosing $\delta={\cal{H}}_{\eps}(t)^{-1}.$
\end{proof}
\medskip

Now we state a "quasi-Lipschitz" property for the field $E_\eps$, which is a modification of a standard inequality (see, 
e.g., \cite{Livrejaune}).
\begin{prop}\label{prop:al-lip}
We have for $t\in [0,T]$ and $x,y\in \R^2$ 
\begin{equation}
 \label{est:al-lip}
 |E_\eps(x,t)-E_\eps(y,t)|\leq C\varphi(|x-y|)\big(\mathcal{H}_\eps(t)+\ln_-|x-y|\big)
\end{equation}
where 
\begin{equation*}
  \varphi(r)=r(\ln_- r+1).
 \end{equation*}
\end{prop}
                       \begin{proof}
 Let $d=|x-y|$. If $d>1/3$ we apply Proposition \ref{coro:E} to obtain 
\begin{equation}
 |E_\eps(x,t)-E_\eps(y,t)|\leq2 \|E_\eps(t)\|_{L^\infty}\leq C\sqrt{\ln\mathcal{H}_\eps(t)}\leq Cd\sqrt{\ln\mathcal{H}_\eps(t)}. \label{b}
 \end{equation}
 Otherwise we set $\bar{z}=(x+y)/2$ and we make the following decomposition:
\begin{equation*}
 \begin{split}
  |E_\eps(x,t)-E_\eps(y,t)|\leq I_1(x,y;t)+I_2(x,y;t)+I_3(x,y;t),
 \end{split}
\end{equation*}
where
\begin{equation*}
 \begin{split}
  I_1(x,y;t)&=\int_{|z-\bar{z}|\leq 2d}\left( \frac{1}{|x-z|}+\frac{1}{|y-z|}\right)\rho_\eps(z,t)\,dz\\
I_2(x,y;t)&=\int_{2d<|z-\bar{z}|< \frac{1}{d}}\left|\frac{1}{|x-z|}-\frac{1}{|y-z|}\right|\rho_\eps(z,t)\,dz\\
I_3(x,y;t)&=\int_{|z-\bar{z}|\geq \frac1d}\left(\frac{1}{|x-z|}+\frac{1}{|y-z|}\right)\rho_\eps(z,t)\,dz.
 \end{split}
\end{equation*}
By \eqref{uu}, proceeding as in \eqref{int} we get
\begin{equation*}
\begin{split}
 I_1(x,y;t)&\leq   C \int_{|z-x|\leq 3d} \frac{1}{|x-z|}\left({\cal{H}}_\eps(t)+\ln_-| z- \xi_\eps(t)|\right)\,dz\\
&\quad \quad +C\int_{|z-y|\leq 3d} \frac{1}{|y-z|}\left({\cal{H}}_\eps(t)+\ln_-| z- \xi_\eps(t)|\right)\,dz\\
&\leq Cd{\cal{H}}_\eps(t)+C\int_{|z|\leq 3d}\frac{\ln_-| z|}{|z|}\,dz, 
\end{split}
\end{equation*}
thus
\begin{equation}\begin{split}
I_1(x,y;t)&\leq Cd\big({\cal{H}}_\eps(t)+\ln_-d\big).
\end{split}
\label{i1}\end{equation}

For the second integral we write, always by \eqref{uu}:
\begin{equation*}\begin{split}
I_2(x,y;t)&\leq Cd\int_{2d<|z-\bar{z}|< \frac{1}{d}}\frac{\rho_\eps(z,t)}{|z-\bar{z}|^2}\ dz\\
&\leq Cd\int_{2d<|z-\bar{z}|< \frac{1}{d}}\frac{{\cal{H}}_\eps(t)+\ln_-| z- \xi_\eps(t)|}{|z-\bar{z}|^2}\ dz\\
&\leq Cd\,{\cal{H}}_\eps(t)\ln_-d
+Cd\int_{2d<|z-\bar{z}|< \frac{1}{d}}\frac{\ln_-| z- \xi_\eps(t)|}{|z-\bar{z}|^2}\ dz.
\end{split}\label{i12}
\end{equation*}

We further split the integral above into two parts,
obtaining on the one side
\begin{equation*}
\int_{\substack{2d<|z-\bar{z}|< \frac{1}{d}
\\| z- \xi_\eps(t)|\geq d}} \frac{\ln_-| z- \xi_\eps(t)|}{|z-\bar{z}|^2}\,dz\leq \ln_- d\int_{2d<|z-\bar{z}|
< \frac{1}{d}} \frac{1}{|z-\bar{z}|^2}\,dz\leq C(\ln_- d)^2,
\end{equation*}
and on the other side
\begin{equation*}
\int_{\substack{2d<|z-\bar{z}|< \frac{1}{d}\\| z- \xi_\eps(t)|< d}} \frac{\ln_-| z- \xi_\eps(t)|}{|z-\bar{z}|^2}\,dz
\leq \frac{1}{4d^2}\int_{|z-\xi_\eps(t)|< d} \ln_-| z- \xi_\eps(t)|\,dz\leq C\ln_- d.
\end{equation*}
Hence, gathering both estimates we get:
\begin{equation*}
\int_{2d<|z-\bar{z}|< \frac{1}{d}} \frac{\ln_-| z- \xi_\eps(t)|}{|z-\bar{z}|^2}\,dz\leq C\big((\ln_- d)^2+\ln_- d\big),
\end{equation*}
whence
\begin{equation}\begin{split}
I_2(x,y;t)&\leq   C\big({\cal{H}}_\eps(t)d\ln_-d+d\ln_-^2d\big). \end{split}
\label{i2}
\end{equation}

Finally for the term $I_3$ we observe that if $|z-\bar{z}|\geq 1/d$ then $\min\{|x-z|,|y-z|\}\geq 1/(2d)$. Hence:
\begin{equation}
I_3(x,y;t)\leq Cd\int\rho_\eps(z,t)\,dz= Cd.\label{i3}
\end{equation}
Estimates \eqref{i1}, \eqref{i2} and \eqref{i3} imply the thesis.

\end{proof}

\medskip

Before presenting the proof of Theorem \ref{thm:main2} we state a useful identity:
\begin{lemma}\label{lem}
\begin{equation*}
\begin{split}
\int_0^t\ &v_\eps(s)\cdot E_\eps(x_\eps(s),s) \ ds\\
&= \Phi_\eps(x_\eps(t),t)-\Phi(x,0)+\int_0^t\, \int  f_\eps(y,w,s)\ w\cdot \frac{x_\eps(s)-y}{|x_\eps(s)-y|^2}\ dy\,dw\,ds,
\end{split}
\end{equation*}
where $\Phi_\eps$ is the potential due to the plasma, that is 
\begin{equation*}
\Phi_\eps(x,t)=\int \ln|x-y|\ \rho_\eps(y,t)\,dy.
\end{equation*}
\end{lemma}
\begin{proof}
Since $E_\eps(x,t)=\nabla_x\Phi_\eps(x,t)$ we have:
$$
 v_\eps(t)\cdot E_\eps(x_\eps(t),t)=\frac{d}{dt}\Phi_\eps(x_\eps(t),t)-\frac{\partial}{\partial t}\Phi_\eps(x_\eps(t),t),
 $$
which implies:
\begin{equation}
\begin{split}
\int_0^t  
 v_\eps(s)\cdot E_\eps(x_\eps(s),s)\,ds=
\Phi_\eps(x_\eps(t),t)-\Phi(x,0)-\int_0^t \frac{\partial}{\partial s}\Phi_\eps(x_\eps(s),s)\,ds.
\label{vE}\end{split}
\end{equation}
We next evaluate the partial derivative in \eqref{vE}. For any $z\in \R^2$, we have
\begin{equation*}
\begin{split}
\frac{\partial}{\partial s}\Phi_\eps(z,s)&=
\frac{\partial}{\partial s}\int f_\eps(y,w,s)\ln|z-y|\,dy\, dw \\
&=\frac{\partial}{\partial s}\int \ f_0(y,w)\ln|z-y_\eps(s)|\,dy\, dw \\
&=-\int f_0(y,w)\ w_\eps(s)\cdot \frac{z-y_\eps(s)}{|z-y_\eps(s)|^2}\, dy\, dw.
\end{split}\label{vEv}
\end{equation*}
Hence the thesis follows.\end{proof}

\medskip

Now we are in position to present the
\medskip

{\it{Proof of Theorem 2.}}

Let $(x,v)\in S_0^{\beta}$ and consider the characteristic $(x_\eps(t),v_\eps(t))$ 
 starting at time $t=0$ from $(x,v)$ and its relative energy $h_\eps(x_\eps(t),v_\eps(t),t)$ for $t\in [0,T].$ 
We compute
\begin{equation}
\frac{d}{dt}h_\eps(x_\eps(t),v_\eps(t),t)=( v_\eps(t)-\eta_\eps(t))\cdot (E_\eps(x_\eps(t),t)+E_\eps(\xi_\eps(t),t))
\label{hh}
\end{equation}
and we see that the singular part disappears in the derivative. We next look for an estimate for  
the time derivative of $h_\eps$ in terms of quantities regarding the plasma which we have all already estimated.

Let $\delta(t)$ be a continuous function, to be chosen later, such that $0\leq\delta(t)<1$. 
Assume that at time  $t\in [0,T]$ we have  $|x_\eps(t)-\xi_\eps(t)|> \delta(t)$.  Then, recalling that
 $\sup_{t\in[0,T]}|\eta_\eps(t)|\leq C_4$ (see \eqref{kin}), we infer from Proposition \ref{coro:E} and from 
\eqref{u} that
\begin{equation}
\label{hh0}\begin{split} 
\left|\frac{d}{dt}h_\eps(x_\eps(t),v_\eps(t),t)\right|&\leq 
2 |v_\eps(t)-\eta_\eps(t)|\|E_\eps(t)\|_{L^{\infty}}\\&\leq C \sqrt{\mathcal{H}_\eps(t)+\ln_- \delta(t)} \sqrt{\ln\mathcal{H}_\eps(t)}.
\end{split}
\end{equation}

If on the contrary $|x_\eps(t)-\xi_\eps(t)|\leq\delta(t)$ then we write eqn. \eqref{hh} as:
\begin{equation}
\begin{split}\frac{d}{dt}h_\eps(x_\eps(t),v_\eps(t),t)=( v_\eps(t)-\eta_\eps(t))
\cdot& (E_\eps(\xi_\eps(t),t)-E_\eps(x_\eps(t),t))\\+2( v_\eps(t)&-\eta_\eps(t))\cdot E_\eps(x_\eps(t),t).
\end{split}\label{hh1}
\end{equation}

We start by estimating the first term in \eqref{hh1}, for which 
we can use the quasi-Lipschitz property stated in Proposition \ref{prop:al-lip} that is:
\begin{equation*}
\begin{split} 
| &v_\eps(t)-\eta_\eps(t)|| E_\eps(\xi_\eps(t),t)-E_\eps(x_\eps(t),t)|\\&\leq
C \sqrt{\mathcal{H}_\eps(t)
+\ln_- |x_\eps(t)-\xi_\eps(t)|}| E_\eps(\xi_\eps(t),t)-E_\eps(x_\eps(t),t)|\\
&\leq C\left( \mathcal{H}_\eps(t)
+\ln_- |x_\eps(t)-\xi_\eps(t)|\right)^{3/2}\varphi(|x_\eps(t)-\xi_\eps(t)|)\\
&\leq
C\mathcal{H}_\eps(t)^{3/2}\varphi(|x_\eps(t)-\xi_\eps(t)|)+C\left(\ln_-|x_\eps(t)-\xi_\eps(t)|\right)^{3/2}
\varphi(|x_\eps(t)-\xi_\eps(t)|).
\end{split} 
\end{equation*}
Now, by definition of the function $\varphi$ we have $\varphi(r) \leq \varphi(\delta(t))$ for $0\leq r\leq \delta(t)$, and
moreover
$$
\varphi(r)(\ln_-r)^{p}<C \quad \forall p> 0.
$$
Therefore we obtain
\begin{equation}
| v_\eps(t)-\eta_\eps(t)|| E_\eps(\xi_\eps(t),t)-E_\eps(x_\eps(t),t)|
\leq C\left(\mathcal{H}_\eps(t)^{3/2} \varphi(\delta(t))+1\right).
\label{11}\end{equation}
Hence from \eqref{hh0}, \eqref{hh1} and \eqref{11} it follows that:
\begin{equation}
\begin{split}
|h&(x_\eps(t),v_\eps(t),t)|\\
&\leq  
\mathcal{H}(0)+C \int_0^t \ \sqrt{\mathcal{H}_\eps(s)
+\ln_- \delta(s)} \sqrt{\ln\mathcal{H}_\eps(s)}\ ds\\
& +
C\int_0^t   \left(\mathcal{H}_\eps(s)^{3/2} \varphi(\delta(s))+1\right)\ ds +2\left|\int_0^t ( v_\eps(s)-\eta_\eps(s))\cdot E_\eps(x_\eps(s),s)\ ds\right|.
\end{split}
\label{hh2}\end{equation}
The last integral on the right-hand side
can be estimated by \eqref{kin}, Proposition \ref{coro:E} and the preceding Lemma \ref{lem}, obtaining:
\begin{equation*}
\begin{split}
\Big|\int_0^t &( v_\eps(s)-\eta_\eps(s))\cdot  E_\eps(x_\eps(s),s)\ ds\Big|
\\&\leq  \left|\int_0^t  v_\eps(s)\cdot E_\eps(x_\eps(s),s)\ ds\right|
+C\int_0^t\sqrt{\ln\mathcal{H}_\eps(s)}\ ds\\
&\leq 
C\int_0^t \sqrt{\ln\mathcal{H}_\eps(s)}\ ds+\big|\Phi_\eps(x_\eps(t),t)-\Phi(x,0)\big|\\
&+\left|\int_0^t \int  f_\eps(y,w,s)\ w\cdot \frac{x_\eps(s)-y}{|x_\eps(s)-y|^2}\,dy\,dw\, ds\right|.
\end{split}
\label{hh3} \end{equation*}
This by Proposition \ref{adjoined} implies:
\begin{equation}
\begin{split}
&\left|\int_0^t\ ( v_\eps(s)-\eta_\eps(s))\ \cdot E_\eps(x_\eps(s),s)\ ds\right| \\
&\quad \leq \left|\Phi_\eps(x_\eps(t),t)\right|+\left|\Phi(x,0)\right|+C \int_0^t  C\sqrt{ {\cal{H}}_{\eps}(s)\ln  {\cal{H}}_{\eps}(s) }\ ds. \label{hh3}
\end{split}
\end{equation}

It remains to estimate the potential terms in \eqref{hh3}. Arguing as in \eqref{P0}, we infer from 
 Remark \ref{rem:3}, \eqref{dens1}, \eqref{elem} and Cauchy-Schwarz inequality 
that $|\Phi(x,0)|\leq C.$ Moreover by \eqref{rho^2} and Proposition  \ref{prop:unif}:
\begin{equation}
\int \ \rho_\eps(y,t)\ln_-|x_\eps(t)-y|\ dy\leq C. \label{hh4}
\end{equation}
On the other side estimates \eqref{mom} and \eqref{suppro} give:
\begin{equation*}
\begin{split}
\int_{|x_\eps(t)-y|\geq 1}  \rho_\eps(y,t)& \ln|x_\eps(t)-y|\ dy\leq\int  \rho_\eps(y,t)(|x_\eps(t)|+|y|)\ dy\\
&\leq 
C+|x_\eps(t)|\int  \rho_\eps(y,t)\ dy \leq C \sqrt{\mathcal{H}_\eps(t)},
\end{split}
\end{equation*}
which, together with estimate \eqref{hh4}, prove that 
\begin{equation}
\left|\Phi_\eps(x_\eps(t),t)\right|\leq C\left(\sqrt{ \mathcal{H}_\eps(t)}+1\right)\leq C\sqrt{ \mathcal{H}_\eps(t)}.
\label{PP1}\end{equation}
The use of \eqref{PP1} in \eqref{hh3} yields
\begin{equation}
\begin{split}
&\left|\int_0^t ( v_\eps(s)-\eta_\eps(s))\ \cdot E_\eps(x_\eps(s),s)\ ds\right| \\
&\quad \leq 
C\sqrt{ \mathcal{H}_\eps(t)}+C \int_0^t   C\sqrt{ {\cal{H}}_{\eps}(s)\ln  {\cal{H}}_{\eps}(s) }\ ds.
\label{vvE}
\end{split}
\end{equation}
Now, inserting  \eqref{vvE} in  \eqref{hh2} we have for any $(x,v)\in S_0^{\beta}$:
\begin{equation*}
\begin{split}
|h(x_\eps(t),v_\eps(t),t)|&\leq  \ \mathcal{H}(0)+C\sqrt{ \mathcal{H}_\eps(t)}\\
&+C\int_0^t\sqrt{\mathcal{H}_\eps(s)
+\ln_- \delta(s)}\sqrt{\ln\mathcal{H}_\eps(s)} \ ds\\
&+  C\int_0^t \Big(\mathcal{H}_\eps(s)^{3/2} \varphi(\delta(s))+ 
\sqrt{ {\cal{H}}_{\eps}(s)\ln  {\cal{H}}_{\eps}(s) }\Big)\,ds .\end{split}
\end{equation*}
Finally, taking the supremum over $(x,v)\in S_0^{\beta}$ we conclude that
\begin{equation*}
\begin{split}
\mathcal{H}_\eps(t)&\leq   \ \mathcal{H}(0)+C\sqrt{ \mathcal{H}_\eps(t)}\\
&+C\int_0^t 
\sqrt{\ln\mathcal{H}_\eps(s)\ln_- \delta(s)}\ ds\\
&+C\int_0^t  \Big(\mathcal{H}_\eps(s)^{3/2} \varphi(\delta(s))+ \sqrt{ {\cal{H}}_{\eps}(s)\ln  {\cal{H}}_{\eps}(s) }\Big)\, ds
.\end{split}\label{qfin}
\end{equation*}
 By choosing
\begin{equation*}
 \delta(t)=\mathcal{H}(t)^{-3/4}
\end{equation*}
 we are led to:
 \begin{equation}
\begin{split}
\mathcal{H}_\eps(t)& \leq  \mathcal{H}(0)
+C\sqrt{ \mathcal{H}_\eps(t)}\\
&+C\int_0^t  \left( 
\ln \mathcal{H}_\eps(s)+ \mathcal{H}_\eps(s)^{3/4}\ln \mathcal{H}_\eps(s)+ \sqrt{ {\cal{H}}_{\eps}(s)
\ln  {\cal{H}}_{\eps}(s) }\right)\,ds\\
&\leq  \mathcal{H}(0)+C\sqrt{ \mathcal{H}_\eps(t)}
+C\int_0^t  \mathcal{H}_\eps(s) \ ds.
\end{split} \label{H}\end{equation}
If the constant $C_1$ in the definition \eqref{h'} of $\mathcal{H}_\eps(t)$ is large enough, we finally obtain:
\begin{equation}
\mathcal{H}_\eps(t)\leq C\mathcal{H}(0)+C\int_0^t  \mathcal{H}_\eps(s)\ ds\end{equation}
with the constants not depending on $\eps$ and $\beta$ and the conclusion follows from Gronwall's Lemma.  \qed

\medskip

Setting $S^\eps_t$ for the support of the density $f_\eps(t)$, the following 
corollary is a direct consequence of Theorem \ref{thm:main2}.

\begin{corollaire} \label{cor:1}There exist positive constants independent of $\eps$ and $\beta$ for which it holds:

\begin{equation}
S_t^\eps=\left\{(x,v)\in \R^2\times \R^2\: :|x|\leq C, \:\mathcal{H}_\eps(t)\leq C \right\}
 \label{supp}
\end{equation}

\begin{equation}
\rho_\eps(x,t)\leq C(1+\ln_- |x-\xi_\eps(t)|). \label{r}
\end{equation}

Moreover:
\begin{equation}
\sup_{t\in [0,T]}\|E_\eps(t)\|_{L^{\infty}}\leq C \label{E}
\end{equation}
and for any  $(x,y)\in S_0^{\beta}:$
\begin{equation}
\sup_{t\in [0,T]} |E_\eps(x,t)-E_\eps(y,t)|\leq C\gamma(|x-y|)
\label{EE}
\end{equation}
where
 \begin{equation}
 \gamma(r)=r(2+\ln_-r)^2 \ \ \ \hbox{if}\ \ \ 0\leq r\leq 1; \ \ \   \gamma(r)= 4 \ \ \ \hbox{otherwise} .\label{var}
 \end{equation}
\end{corollaire}
\begin{proof}
All the bounds are direct consequences of Theorem \ref{thm:main2}; Estimate \eqref{supp} follows from \eqref{suppro}
 in Lemma \ref{logar}, \eqref{r} and \eqref{E} follow from \eqref{uu}
 and Proposition \ref{coro:E} respectively. Finally \eqref{EE}
 follows from Proposition \ref{prop:al-lip} in case $|x-y|\leq 1$ and from \eqref{E} otherwise. 
\end{proof}

\medskip
\begin{remark} Notice that the function $\gamma$ is positive, increasing, continuous and concave on $\R^+.$\label{rem:conc}
\end{remark}

\section{Proof of Theorem \ref{thm:main}}
\label{section:compactness}

In this section we prove the convergence of the regularized system introduced in Section \ref{sec:approx}. 
To do this we keep first $\beta$ fixed and state the 
$\eps$-convergence results in the following Propositions \ref{prop:compactness-E} and \ref{char}.  
Then, using the fact that all our estimates are uniform in $\eps$ and $\beta$ we will be able to remove 
also the cutoff $\beta$.

 We consider the solution $\big(x_ {\eps_n}(t),v_ {\eps_n}(t); \xi_ {\eps_n}(t),\eta_ {\eps_n}(t); f_ {\eps_n}(t)\big)$  
to system \eqref{eq:syst1}-\eqref{eq:transport-one}  for some sequence $\eps_n\to 0$ as $n\to \infty$. 
The following result holds.

\begin{prop}
\label{prop:compactness-E} 
There exists a subsequence of $\{\eps_n\}$ (which we still denote by $\{\eps_n\}$) and 
there exists $E\in C(\RR\times [0,T])$ such that $E_{\eps_n}$ converges to $E$
uniformly on the compact sets of $\RR \times [0,T]$.
\end{prop}

\begin{proof}
By \eqref{E} the sequence $E_{\eps_n}$ is uniformly bounded on the compact sets of $\RR \times [0,T]$. 
 Moreover, in view of \eqref{EE}, it is uniformly equicontinuous in $x.$ The result will be a consequence of the 
Ascoli-Arzela's theorem, once we have proven that it is also
 uniformly equicontinuous with respect to time. To this aim we choose $t,s\in[0,T]$ such that $|t-s|<1$ and  we introduce a 
positive, bounded increasing $C^{\infty}$-function $g$ defined as:
 
\begin{equation}
g(r)=1 \ \ \hbox{if}\ \ r>a, \ \ \ g(r)=0  \ \ \hbox{if}\ \ r\leq \frac{a}{2},
\label {g1}
\end{equation}
where $a \in (0,1)$ will be suitably chosen hereafter, with the further property that for some $C>1/2$
\begin{equation}
0\leq g'(r)\leq \frac{C}{a}.
\label {g2}
\end{equation}
Next, we write
\begin{equation*}
 \begin{split}
 E_{\eps_n}(x,t)&-E_{\eps_n}(x,s)\\
&= \int \ [1-g(|x-y|)] (\rho_{\eps_n}(y,t)-\rho_{\eps_n}(y,s))\frac{x-y}{|x-y|^2}\,dy\\
&+\int \ g(|x-y|)(\rho_{\eps_n}(y,t)-\rho_{\eps_n}(y,s))\frac{x-y}{|x-y|^2}\,dy\\
&=  I_1(x;t,s)+I_2(x;t,s).
\end{split}
\end{equation*}
By definition of $g$ and \eqref{r} we have
\begin{equation*}
\begin{split}
|I_1(x;t,s)|&\leq \sup_{\tau\in [0,T]} \int_{|x-y|\leq a} \frac{\rho_{\eps_n}(y,\tau)}{|x-y|}\,dy\leq C\sup_{\tau\in [0,T]}\int_{|x-y|\leq a} \frac{1+\ln_-|y-\xi_{\eps_n}(\tau)|}{|x-y|}\,dy.
\end{split}
\end{equation*}
Hence by \eqref{int-} we obtain:
\begin{equation}
\sup_{0\leq t,s\leq T}\sup_{x\in \RR}|I_1(x;t,s)|\leq C\left( a+a |\ln a|\right)\leq C a |\ln a|\label{1}.
\end{equation}
 
Now we estimate the term $I_2(x;t,s),$ writing 
\begin{equation}
\begin{split}
I_2(x;t,s) =
\int\ f_0(y,w)\int_s^t  \left[ \frac{d}{d\tau}  
\Big( g(|x-y_{\eps_n}(\tau)|) \frac{x-y_{\eps_n}(\tau)}{|x-y_{\eps_n}(\tau)|^2}\Big)\,d\tau\right]\,dy\,dw .
\end{split} \label {g3}
\end{equation}

Since 
\begin{equation*}
\frac{d}{d\tau}  
\frac{x-y_{\eps_n}(\tau)}{|x-y_{\eps_n}(\tau)|^2}= 
-\frac{w_{\eps_n}(\tau)}{|x-y_{\eps_n}(\tau)|^2}+\frac{2(x-y_{\eps_n}(\tau))\cdot w_{\eps_n}(\tau)}{|x-y_{\eps_n}(\tau)|^4}\  
(x-y_{\eps_n}(\tau)),
\end{equation*}
we have
\begin{equation*}
\left|\frac{d}{d\tau}  \frac{x-y_{\eps_n}(\tau)}{|x-y_{\eps_n}(\tau)|^2}\right|= \frac{|w_{\eps_n}(\tau)|}{|x-y_{\eps_n}(\tau)|^2}.
\end{equation*}
Hence from \eqref{g3} and properties \eqref{g1}-\eqref{g2} for $g$ it follows that
\begin{equation*}
\begin{split}
|I_2(x;t,s)|&\leq
\int\ f_0(y,w)\left[\int_s^t  
g(|x-y_{\eps_n}(\tau)|)\frac{|w_{\eps_n}(\tau)|}{|x-y_{\eps_n}(\tau)|^2}\,d\tau\right]  \,dy\,dw\\
&+
\frac{C}{a}\int \ f_0(y,w)\left[\int_s^t \frac{|w_{\eps_n}(\tau)|}{{|x-y_{\eps_n}(\tau)|}}\,d\tau\right]  \,dy\,dw\\
&\leq
\frac{C}{a}\int \ f_0(y,w)\left[\int_s^t 
\frac{|w_{\eps_n}(\tau)|}{{|x-y_{\eps_n}(\tau)|}}\,d\tau\right]  \,dy\,dw\\
&=\frac{C}{a}\int_s^t  \int \ |w|\frac{f_{\eps_n}(y,w,\tau)}{{|x-y|}} \,dy\,dw\,d\tau.
\end{split}
\end{equation*}
Thanks to Proposition \ref{adjoined} and Theorem \ref{thm:main2} we conclude that
\begin{equation}
|I_2(x;t,s)|\leq C\ \frac{|t-s|}{a},\label{1'}
\end{equation}
and the conclusion follows from \eqref{1} and \eqref{1'} by choosing $a= |t-s|^{1/2}$.  
\end{proof}

Proposition \ref{prop:compactness-E}
 is an important step to prove the uniform in time convergence of $\left(\xi_ {\eps_n},\eta_ {\eps_n}\right)$
 and of $\left(x_ {\eps_n}(x,v),v_ {\eps_n}(x,v)\right)$ for almost-every initial configuration $(x,v)$. 
However we need also to control the size of the "bad initial configurations", 
those leading to possible collapses with the charge at some time. This is done in Lemma \ref{prop:mesure} below, the proof
of which is postponed in the Appendix at the end of this section.

\begin{lemma}\label{prop:mesure}
 Let  $0<\delta<1/4$ and set for any $n$
\begin{equation*}
 S_n(\delta)=\left\{(x,v)\in S_0^{\beta}: \min_{t\in [0,T]}|x_{\eps_n}(t)-\xi_{\eps_n}(t)|<\delta\right\}.
\end{equation*}
Setting $|S_n(\delta)|$ to indicate its volume in the phase space, we have
\begin{equation*}
|S_n(\delta)|\leq C\delta |\ln \delta|^{3/2}
\end{equation*}
with $C$ independent of $n$ and $\beta.$
\end{lemma}

\medskip

Finally, the following statement is the last result we need in order to prove Theorem \ref{thm:main}:

 \begin{prop}
The sequences $x_{\eps_n}$ and $v_{\eps_n}$ are Cauchy sequences in $L^1\left(d\mu_0;C([0,T])\right).$
The sequences $\xi_{\eps_n}$ and $\eta_{\eps_n}$ are Cauchy sequences in $C([0,T]).$
\label{char}
\end{prop}
\begin{proof}
Let us set
$$
X_{n,m}(t)= |x_{\eps_n}(t)-x_{\eps_m}(t)|+|\xi_{\eps_n}(t)-\xi_{\eps_m}(t)|.
$$
We have for all $ t\in [0,T]:$
\begin{equation}
 \begin{split}
 X_{n,m}(t)&\leq  \int_0^t \int_0^s |E_{\eps_n}(x_{\eps_n}(\tau),\tau)-
E_{\eps_m}(x_{\eps_m}(\tau),\tau)|\,d\tau\,ds\\
&+  
\int_0^t \int_0^s|E_{\eps_n}(\xi_{\eps_n}(\tau),\tau)-
E_{\eps_m}(\xi_{\eps_m}(\tau),\tau)|\,d\tau\,ds\\
&+
\int_0^t \int_0^s\left|F_{\eps_n}(x_{\eps_n}(\tau),\tau)-F_{\eps_m}(x_{\eps_m}(\tau),\tau)\right|\,d\tau\,ds.
 \end{split}
\label{diff}\end{equation}

Fix now a positive parameter  $\delta$ such that  $2\max(\eps_n,\eps_m)<\delta<1.$
We decompose $S_0^{\beta}$ as
\begin{equation*}
S_0^{\beta}=\Gamma(\delta)\cup \Gamma(\delta)^c,
\end{equation*}
where 
\begin{equation*}
 \Gamma(\delta) =\ S_n(\delta)\cup S_m(\delta)
\end{equation*}
with $S_n(\delta)$ the set introduced in Lemma \ref{prop:mesure}. 
Clearly, for fixed $\delta$ the fields $F_{\eps_n}$ are bounded and Lipschitz uniformly with respect to $n$
 on $\Gamma(\delta)^c$; This make it possible to  handle the last integral in the right-hand side in 
\eqref{diff}. On the other side, Lemma \ref{prop:mesure} provides a control on 
the size of the bad set $\Gamma(\delta)$. Therefore in order to prove
 Proposition \ref{char}  we will first let 
$n,m\to \infty$ for fixed $\delta$, then let $\delta \to 0$.

To that aim we introduce
\begin{equation*}
\begin{split}
 \mathcal{X}_{n,m}^1(t)&=\frac{1}{\mu_0(\Gamma(\delta)^c)}
\int_{\Gamma(\delta)^c} \sup_{s\in [0,t]}X_{n,m}(s)\ d\mu_0(x,v),\\
 \mathcal{X}_{n,m}^2(t)&=\int_{\Gamma(\delta)}
 \sup_{s\in [0,t]}X_{n,m}(s)\ d\mu_0(x,v)
\end{split}
\end{equation*}
and we notice that by Lemma \ref{prop:mesure} $\mu_0(\Gamma(\delta)^c)\to 1$ as $\delta \to 0.$

\medskip

We start by estimating $\mathcal{X}_{n,m}^1(t)$. 
We estimate the first two terms in the right-hand side of \eqref{diff} by observing that
\begin{equation}
\begin{split}
&|E_{\eps_n}(x_{\eps_n}(\tau),\tau)-
E_{\eps_m}(x_{\eps_m}(\tau),\tau)|\\&\leq 
 |E_{\eps_n}(x_{\eps_n}(\tau),\tau)-
E_{\eps_m}(x_{\eps_n}(\tau),\tau)|+ |E_{\eps_m}(x_{\eps_n}(\tau),\tau)-
E_{\eps_m}(x_{\eps_m}(\tau),\tau)|
\end{split}\label{3}
\end{equation}
and by making an analogous decomposition for the second term.
By the quasi-Lipschitz property \eqref{EE} and Remark \ref{rem:conc}, we have 
\begin{equation}
\begin{split}
 \label{ineq:cauchy1}
|E_{\eps_m}(x_{\eps_n}(\tau),\tau)-&
E_{\eps_m}(x_{\eps_m}(\tau),\tau)|+|E_{\eps_m}(\xi_{\eps_n}(\tau),\tau)-
E_{\eps_m}(\xi_{\eps_m}(\tau),\tau)|\\&\leq C\gamma(X_{n,m}(\tau)).
\end{split}
\end{equation}

Next, if $(x,v)$ belongs to $\Gamma(\delta)^c$, we have 
on
$ [0,t]$
\begin{equation*}
\begin{split}
\left|F_{\eps_n}(x_{\eps_n}(\tau),\tau)-F_{\eps_m}(x_{\eps_m}(\tau),\tau)\right|
=
\left|\frac {x_{\eps_n}(\tau)-\xi_{\eps_n}(\tau)}{|x_{\eps_n}(\tau)-\xi_{\eps_n}(\tau)|^2}-
\frac {x_{\eps_m}(\tau)-\xi_{\eps_m}(\tau)}{|x_{\eps_m}(\tau)-\xi_{\eps_m}(\tau)|^2}\right|,\end{split}\end{equation*}
therefore
\begin{equation}\label{ineq:Fn}\begin{split}
\left|F_{\eps_n}(x_{\eps_n}(\tau),\tau)-F_{\eps_m}(x_{\eps_m}(\tau),\tau)\right|&\leq 
\frac{\left|X_{n,m}(\tau)\right|}{\delta^2}.\end{split}
\end{equation}

Going back to \eqref{diff}, in view of \eqref{3}, \eqref{ineq:cauchy1} and \eqref{ineq:Fn} 
we obtain for all $(x,v)\in \Gamma(\delta)^c$
\begin{equation*}
\begin{split}
 &X_{n,m}(t)\\&\leq C \sup_{t\in  [0,T]}\|E_{\eps_n}(t)-E_{\eps_m}(t)\|_{L^{\infty}}
+C\int_0^t \int_0^s\left[\gamma(X_{n,m}(\tau))
+ \frac{X_{n,m}(\tau)}{\delta^2}\right]\,d\tau\,ds\\&
\leq C \sup_{t\in  [0,T]}\|E_{\eps_n}(t)-E_{\eps_m}(t)\|_{L^{\infty}}+
 \frac{C}{\delta^2}\int_0^t \int_0^s\gamma(X_{n,m}(\tau))
\,d\tau\,ds, \end{split}
\end{equation*}
so that, being $\gamma$ an increasing function, it follows 
\begin{equation*}\begin{split}
\mathcal{X}_{n,m}^1(t)&\leq  C\sup_{t\in  [0,T]}\|E_{\eps_n}(t)-E_{\eps_m}(t)\|_{L^{\infty}}\\& +
 \frac{C}{\delta^2}
\int_0^t \int_0^s \left[\frac{1}{\mu_0(\Gamma(\delta)^c)}
\int_{\Gamma(\delta)^c}\gamma\left(\sup_{u \in [0,\tau]}X_{n,m}(u)\right)d\mu_0(x,v) \right]\,d\tau\,ds.
\end{split}\end{equation*}
By concavity of $\gamma$ we may apply the Jensen inequality to the above integral, obtaining:
\begin{equation}
\label{x} \begin{split}
 &\mathcal{X}_{n,m}^1(t)\leq C \sup_{t\in  [0,T]}\|E_{\eps_n}(t)-E_{\eps_m}(t)\|_{L^{\infty}}+\frac{C}{\delta^2} 
\int_0^t \int_0^s\gamma(\mathcal{X}_{n,m}^1(\tau))\,d\tau\,ds.
\end{split}
\end{equation}

Next, we have in view of Proposition \ref{prop:compactness-E}:
\begin{equation}
\lim_{n,m\to \infty}  \sup_{t\in  [0,T]} \|E_{\eps_n}(t)-E_{\eps_m}(t)\|_{L^{\infty}}
=0. \label{lim}
\end{equation}
This, by definition of $\gamma,$ implies that the second order integral inequality \eqref{x} can 
be handled by choosing $n,m$ large in function of $\delta.$ More precisely, defining
$$
\omega(n,m)= \sup_{t\in  [0,T]} \|E_{\eps_n}(t)-E_{\eps_m}(t)\|_{L^{\infty}}
$$
 we obtain 
\begin{equation}
\label{ineq:x1}
\mathcal{X}_{n,m}^1(t)\leq C[\omega(n,m)]^{\exp(-Ct/\delta)} .
\end{equation}
The proof of \eqref{ineq:x1} 
is elementary and is given in Lemma \ref{lemma:gronwall} in the Appendix at the end of this section. 
We want to stress that it is possible to prove it only because \eqref{x} is a second order integral inequality.

\medskip

Let us now estimate the remaining term $\mathcal{X}_{n,m}^2(t)$. In view of the bound \eqref{supp} on 
the support of the spatial density of the plasma and of Lemma \ref{prop:mesure} we get:
\begin{equation}
\label{ineq:x2}
\mathcal{X}_{n,m}^2(t)\leq C\mu_0\left(\Gamma(\delta)\right)\leq C\delta |\ln \delta|^{3/2}.
\end{equation}

Gathering \eqref{ineq:x1} and \eqref{ineq:x2} we obtain
\begin{equation}
 \mathcal{X}_{n,m}^1(t)+\mathcal{X}_{n,m}^2(t)
\leq  C[\omega(n,m)]^{\exp(-Ct/\delta)}+C\delta |\ln \delta|^{3/2}.\label{ineq:x12}
\end{equation}

Hence, defining 
$$
\mathcal{X}_{n,m}(t)=\mathcal{X}_{n,m}^1(t)+\mathcal{X}_{n,m}^2(t),
$$
we conclude that  
\begin{equation*}
\limsup_{n,m\to \infty}\mathcal{X}_{n,m}(t)\leq C\delta |\ln \delta|^{3/2}.
\end{equation*}
This, by arbitrariness of $\delta,$  implies the first part of the thesis, that is
$x_{\eps_n}$ is a Cauchy sequence in $L^1\left(d\mu_0;C([0,T])\right)$ and $\xi_{\eps_n}$ is a Cauchy sequence
in $C([0,T])$.

\medskip

For the velocities we introduce analogous definitions, that is:
\begin{equation*}\begin{split}
V_{n,m}(t)&=|v_{\eps_n}(t)-v_{\eps_m}(t)|+|\eta_{\eps_n}(t)-\eta_{\eps_m}(t)|\\
\mathcal{V}_{n,m}^1(t)&=\frac{1}{\mu_0(\Gamma(\delta)^c)}\int_{\Gamma(\delta)^c}
 \sup_{s\in[0,t]}V_{n,m}(s)\ d\mu_0(x,v)\\
\mathcal{V}_{n,m}^2(t)&=\int_{\Gamma(\delta)}
\sup_{s\in[0,t]}V_{n,m}(s)\ d\mu_0(x,v)
\end{split}
\end{equation*}
and 
$$
 \mathcal{V}_{n,m}(t)= \mathcal{V}_{n,m}^1(t)+ \mathcal{V}_{n,m}^2(t).
 $$
 Proceeding in analogy with the previous computation, we infer from \eqref{ineq:x1} that
\begin{equation}
\label{ineq:v1}\begin{split}
\mathcal{V}_{n,m}^1(t)&\leq  C{\omega}(n,m)
+\frac{C}{\delta^2} \int_0^t\gamma(\mathcal{X}_{n,m}^1(\tau))\,d\tau\\
&\leq C\omega(n,m)+\frac{C}{\delta^2}\gamma\big([\omega(n,m)]^{\exp(-Ct/\delta)}\big).
\end{split}
\end{equation}

On the other hand,  Cauchy-Schwarz inequality combined with the bound on the charge velocity yields:
\begin{equation*}
 \begin{split}
  &\mathcal{V}_{n,m}^2(t)\\
&\leq
C|\Gamma(\delta)|^{1/2}
\left( \int \sup_{s\in[0,t]}\left(|v_{\eps_n}(s)|^2+|v_{\eps_m}(s)|^2\right)f_0(x,v)\,dx\,dv\right )^{1/2}+C|\Gamma(\delta)|\\
&= C|\Gamma(\delta)|^{1/2}
\left( \int |v|^2\sup_{s\in[0,t]}\left(f_{\eps_n}(x,v,s)+f_{\eps_m}(x,v,s)\right)\,dx\,dv\right)^{1/2}+C|\Gamma(\delta)|.
 \end{split}
\end{equation*}
Therefore, thanks to the bound \eqref{kin} on the kinetic energy and to Lemma \ref{prop:mesure} again we obtain
\begin{equation}
 \label{ineq:v2}
\mathcal{V}_{n,m}^2(t)\leq C \delta^{1/2}|\ln \delta|^{3/4}.
\end{equation}
Again, in view of \eqref{lim}, \eqref{ineq:v1} and \eqref{ineq:v2} we are led to
\begin{equation*}
\limsup_{n,m\to \infty}\mathcal{V}_{n,m}(t)\leq C\delta^{1/2}|\ln \delta|^{3/4}
\end{equation*}
and the conclusion follows as before.
Hence the proof of Proposition \ref{char} is complete.
\end{proof}

\medskip
The results achieved up to now allow us to complete the 
\medskip

{\it{Proof of Theorem \ref{thm:main}.}}

Thanks to Proposition \ref{char}, there exists  $(\xi(\cdot),\eta(\cdot))$ and, 
for $d\mu_0$-a.a. $(x,v)$, there exists  $(x(x,v,\cdot),v(x,v,\cdot))$ such that
$ (\xi_{\eps_n}(\cdot),\eta_{\eps_n}(\cdot))$ converges to $(\xi(\cdot),\eta(\cdot))$ and, for  $d\mu_0$-a.a. $(x,v)$,
$(x_{\eps_n}(x,v,\cdot),v_{\eps_n}(x,v,\cdot))$ converges to $(x(x,v,\cdot),v(x,v,\cdot))$ 
 uniformly on $[0,T]$. It follows that the map $(x,v)\mapsto \left(x(x,v,t),v(x,v,t)\right)$ preserves the Lebesgue's measure
on $S_0^\beta$ for all $t\in [0,T]$.

Next,  for $t\in [0,T]$ we define 
the measure $d\mu(t)=\left(x(\cdot,\cdot,t),v(\cdot,\cdot,t)\right)_{\#} d\mu_0$ on $\R^2\times \R^2$, i.e.
\begin{equation}\label{eq:transport}
 \int \varphi(x,v)\,d\mu(t)=\int \varphi\left(x(t),v(t)\right)f_0(x,v)\,dx\,dv\quad  \forall 
\varphi\in C_b(\R^2\times \R^2).
\end{equation}

Clearly, by \eqref{eq:transport-one} and \eqref{eq:transport} the sequence $f_{\eps_n}(t)$ converges to $d\mu(t)$ in the weak sense of measures for all $t\in [0,T]$.  On
the other hand,  $f_{\eps_n}(t)$ is transported by a measure-preserving flow and it satisfies the uniform bounds
 \eqref{kin} and \eqref{supp}; therefore, it is also uniformly equi-integrable and Dunford-Pettis theorem 
ensures that $f_{\eps_n}(t)$ is weakly relatively compact in $L^1$. As a result, we have $d\mu(t)=f(t)\,dx\,dv$
for some $f(t)\in L^1$. Since, on the other hand, $f_{\eps_n}$ is uniformly bounded in $L^\infty\left([0,T],L^1\cap L^\infty\right)$
we conclude that $f\in L^\infty\left([0,T],L^1\cap L^\infty\right)$.

We may now define $\rho=\int f\,dv\in L^\infty(L^1)$.
By the same arguments, we check that $\rho_{\eps_n}(t)$ converges weakly in $L^1$ to $\rho(t)$ for $t\in [0,T]$; 
in particular the bound \eqref{p1} in Theorem \ref{thm:main} holds for $\rho$. Setting then $E=\rho\ast x/|x|^2$ and using the fact that
$\rho$ satisfies \eqref{p1}, we obtain, mimicking (for example) the arguments of the proof of Proposition
\ref{coro:E}, that  $E_{\eps_n}$ converges to $E$
uniformly on $\RR \times [0,T]$.

Furthermore, Lemma \ref{prop:mesure} 
shows that the set $\{(x,v)\in S_0^\beta: \inf_{t\in [0,T]}|x(t)-\xi(t)|=0\}$ 
has zero $d\mu_0$-measure. This ensures uniform convergence of the singular field 
$F_{\eps_n}( x_{\eps_n},\cdot)$ for  $d\mu_0$-a.a. initial data. We can then pass to the limit in 
equations \eqref{eq:syst1} to find that  $(x(t),v(t);\xi(t),\eta(t);f(t))$ 
satisfy eqns. \eqref{eq:system} and \eqref{ch} over $[0,T]$. In particular, \eqref{inv} follows from 
\eqref{eq:transport} and from the fact that $(x,v)\mapsto (x(x,v,t),v(x,v,t))$ preserves Lebesgue's measure. 

\medskip

As a matter of fact, we proved  Theorem \ref{thm:main} only for  $f_0^{\beta}$ being supported in $S_0^\beta$.  
However, for any $f_0$ supported in $S_0$, let us introduce the sequence of initial conditions
$$
f_0^\beta(x,v)={\cal N}_\beta f_0(x,v) \chi ( |x-\xi|> \beta) 
$$
where ${\cal N}_\beta$ is a normalization factor since we are working with probability distribution. 
Obviously ${\cal N}_\beta \to 1 $ as $\beta \to 0$. Let $f^\beta(x,v,t)$ be a corresponding sequence of
solutions we have already constructed. Since all our estimates are uniform in $\beta$, we can 
remove the   $\beta$ cut-off exactly the same way as we just did with $\eps$, achieving thereby the proof of  Theorem \ref{thm:main}.

Actually we could also have proven our result  in a different way, working with a single sequence, 
by choosing $\beta=\beta (\eps)$  suitably vanishing with $\eps$.

\subsection*{Appendix}

\textbf{Proof of Lemma \ref{prop:mesure}}

\medskip

For $0<\delta <1/4$ we set $S_n(\delta)=S_n^1(\delta)\cup S_n^2(\delta)$, where 
\begin{equation*}\begin{split}
S_n^1(\delta)&=\{ (x,v)\in S_n(\delta):x\in B(\xi,2\delta)\} \\
S_n^2(\delta)&=\{ (x,v)\in S_n(\delta):x\in B(\xi,2\delta)^c\}.
\end{split}\end{equation*}
The estimate on the measure of $S^1_n(\delta)$ is trivially given by the hypothesis \eqref{beta} 
on the initial data and its consequence \eqref{v0}:
\begin{equation}
|S^1_n(\delta)|\leq  \int_{x\in B(\xi,2\delta)} \chi(S_0^{\beta})\, dx\, dv\leq C\int_{x\in B(\xi,2\delta)}\ 
\ln_-|x-\xi|\ dx\leq C\delta^2 |\ln \delta|. \label{s1}
\end{equation}
On the other hand, let $(x,v)\in S^2_n(\delta).$ By continuity there exists $t_0=t_0(x,v)$ such that
$|x_{\eps_n}(t_0)-\xi_{\eps_n}(t_0)|=\delta$. We set $(t^-,t^+)\in [0,T]$ for the connected component 
containing $t_0$ such that 
$\delta/2<|x_{\eps_n}(t)-\xi_{\eps_n}(t)|<2\delta$ for $t\in (t^-,t^+).$  By virtue of \eqref{u}  and 
Theorem \ref{thm:main2} for any $(x,v)\in S^2_n(\delta)$ we then have
\begin{equation*}
 |v_{\eps_n}(t)|^2\leq C|\ln \delta|,\quad \forall t\in (t^-,t^+),
\end{equation*}
therefore 
$$|t^+-t^-|\geq C\frac{\delta}{\sqrt{|\ln \delta|}}:=\Delta T.$$
Now we partition the interval  $[0,T]$ into $N(\delta)+1$ intervals $[t_i,t_{i+1}]$ of length smaller than $\Delta T/2$. 
Then $(t^-,t^+)$ has to contain at least one of the $t_i$, so that
\begin{equation*}
 S_n^2(\delta)\subset \bigcup_{i=0}^{N(\delta)} S_{n,i}(\delta),
\end{equation*}
 where 
\begin{equation*}
 S_{n,i}(\delta)=\big\{(x,v)\in S_0^\beta:\quad \delta/2<|x_{\eps_n}(t_i)-\xi_{\eps_n}(t_i)|<2\delta\big \}.
\end{equation*}
Next, since the flow $(x,v)\mapsto (x_{\eps_n}(x,v,t_i),v_{\eps_n}(x,v,t_i))$ preserves the
 Lebesgue's measure on $\R^2\times \R^2$ we have
 \begin{equation*}
| S_{n,i}(\delta)|=\big|\big\{(x,v)\in S_{n}(t_i):\delta/2<|x-\xi_{\eps_n}(t_i)|<2\delta \big\}\big|\end{equation*}
where $S_{n}(t_i)$ is the support of $f_{\eps_n}(t_i)$.
Thus, in analogy with estimate \eqref{s1}, by Theorem \ref{thm:main2} we get:
$$
| S_{n,i}(\delta)|\leq C\delta^2|\ln \delta| .
$$
Since by construction $N(\delta)\Delta T\leq CT$, we finally obtain
\begin{equation}
 |S_n^2(\delta)|\leq\sum_{i=0}^{N(\delta)} |S_{n,i}(\delta)|\leq C N(\delta)\delta^2 |\ln \delta|
\leq C\delta |\ln \delta|^{3/2}, \label{s2}
\end{equation}
and the conclusion follows from \eqref{s1} and \eqref{s2}.

\qed

\medskip

\begin{lemma} \label{lemma:gronwall}  

 Let $u$ be a positive, continuous function on $[0,T]$ such that
 $$
 u(t)\leq b+a\int_0^t\int_0^s \gamma (u(\tau))\ d\tau ds
 $$ 
 where 
$ \gamma$ is the function defined in \eqref{var}, $a>0$ and $0<b<1$ such that
\begin{equation*}
  b< \exp(1-e^{2T\sqrt{a}}).
\end{equation*}
Then
$$
u(t)\leq v(t)
$$
where the function $v(t)$ is the solution to the differential equation
\begin{equation*}
\begin{split}
&\ddot{v}=a \ \gamma (v)\\ &v(0)=b \quad \dot{v}(0)=0.\end{split}
\end{equation*}
Moreover:
\begin{equation}
v(t)\leq Cb^{\exp(-2T\sqrt{a})}. \label{lemma}
\end{equation}
\end{lemma}
\begin{proof}
The proof of the first part of the statement is standard so that we are left with the 
proof of \eqref{lemma}. By the properties of the function $\gamma$ there exists $T^{\ast}\leq T$ maximal such that $b<v(t)<1$ on $(0,T^\ast)$. By multiplying both terms in the differential equation
 by $\dot{v}$
we get for $t\in (0,T^\ast)$
\begin{equation*}
\frac{1}{2} \frac{d}{dt}(\dot{v}^2)= a\frac{d}{dt}\phi(v),
\end{equation*}
where $\phi$ is a primitive of $\gamma$. Notice that, by the definition of $\gamma,$ $\phi$ is an increasing function such that
 $\phi(v)\leq  2v^2(1-\ln v)^2$ $\forall v\in[0,1]$. Therefore
 \begin{equation*}
  \dot{v}\leq \sqrt{2a(\phi(v)-\phi(b))}\leq \sqrt{2a\phi(v)}\leq 2\sqrt{a} \ v(1-\ln v).
 \end{equation*}
Since the function $v(1-\ln v)$ is positive and increasing in $(0,1]$ we can apply the Gronwall lemma to this differential inequality, getting:
\begin{equation*}
v(t)\leq \exp(1-e^{-2T\sqrt{a}})b^{\exp(-2T\sqrt{a})},\quad t\in (0,T^\ast).
\end{equation*}
By the choice of $b$ and the definition of $T^\ast$, we obtain $T^*=T$,
so that the previous inequality holds on $[0,T]$ and the Lemma is proved.
\end{proof}

\section{The case of $N$-charges}\label{N}

This section is devoted to the system already presented in the introduction, consisting of $N$ 
negative point charges and a positive plasma. Setting $(\xi^i ( t),  \eta^i (t))$ for position and 
velocity of the $i$-th charge at time $t$ with initial condition $( \xi^i,  \eta^i),$  
the equations \eqref{VP0}-\eqref{VP2} in terms of characteristics are:
\begin{equation}
\label{eq:NN}
\begin{cases}
\dsp\dot{x}(t) = v(t)\\ 
\dsp\dot{v}(t)=E\left( x(t),t\right)-F\left( x(t),t\right) \\
\dsp x(0)=x, \ \ \ v(0)=v \\
\dsp\dot{\xi^i}(t)= \eta^i(t) \\ 
\dsp \dot{\eta^i}(t)=- E\left({\xi}^i(t),t\right) + \sum_{ j\neq i} F^{j}\left( \xi^i(t),t\right)\\ 
\dsp \xi^i(0)=\xi^i, \ \ \  \eta^i (0)=\eta^i,\quad i=1,\ldots,N\\
\dsp f\left(x(t),v(t),t\right)=f_0(x,v)
\end{cases}
\end{equation}
 where $E$, $F$ and $ F^j$ have been defined in \eqref{VP00} and \eqref{VP1}. 

In order to present our global existence result we need some notations to 
describe the support of the initial density $f_0$ in system \eqref{eq:NN}. Let $d_0$ be the minimal distance 
between two charges at time $t=0.$
We set
$$
\Lambda_i=B\left(\xi^i,\frac{d_0}{4}\right)\times \R^2, \quad \Lambda^c=\bigcap_{i=1}^N\Lambda_i^c
$$
where $B(\xi^i,d_0/4)=\{x:|x-\xi^i|\leq d_0/4\}.$ 

We introduce the energy of a plasma particle  relative to the $i$-th charge:
\begin{equation*}
 h^i(x,v,t)=\frac{1}{2}|v-\eta^i(t)|^2+\ln|x-\xi^i(t)|.
\end{equation*}

We further define
\begin{equation*}
h_0(x,v)=\begin{cases}\dsp |h^i(x,v,0)|\quad \text{if } (x,v)\in \Lambda_i,\\
\dsp \max_i |h^i(x,v,0)|\quad \text{if } (x,v)\in \Lambda^c.
 \end{cases}
\end{equation*}

 We assume that   
the support of $f_0$ is the set given by
\begin{equation}
\begin{split}
&S_0=S_0\left(C_0, \frac{d_0}{4}\right)=\Big\{(x,v): h_0(x,v)\leq C_0\Big\}\end{split} \label{max}
\end{equation}
for some positive $C_0$.

\begin{remark}
In order to have a finite $C_0$ in definition \eqref{max} 
we need to decompose the set of initial data by means of  the sets $\Lambda_i$ and $\Lambda^c.$ 
This is due to the fact that whenever a particle is close to the $i$-th charge, then
 its velocity has to be large because $h^i(x,v,0)$ is assumed to be bounded. 
On the other hand, for the same particle, $h^j(x,v,0)$ would diverge for any $j\neq i$, 
since it is far from the $j$-th charge and its potential part cannot compensate such a  large velocity.
\end{remark}  

We are now in position to state the main result of this section:
\begin{thm}\label{thm:mainN}
Let $\xi^1,\ldots,\xi^N$ be distinct points of $\R^2$, $\eta^1,\ldots,\eta^N\in \R^2$ 
and $f_0\in L^\infty(\R^2\times \R^2)$ be a probability density supported on the set $S_0(C_0,d_0/4)$ for some positive
$C_0$.
Let $T>0$. Then there exists 
\begin{equation*}
 \begin{split}
&f\in L^\infty\left([0,T]; L^\infty\cap L^1(\R^2\times \R^2)\right),\quad 
E\in L^\infty\left([0,T];L^\infty(\R^2)\right),\\
& (\xi^i(t\cdot),\eta^i(\cdot))\in C^1([0,T])^2,\quad i=1,\ldots,N,\\
&(x(\cdot),v(\cdot))\in C^1([0,T])^2\quad \text{for $d\mu_0$-a.a. $(x,v)\in S_0(C_0,d_0/4)$}
\end{split}\end{equation*} 
 such that  $\big(x(t),v(t);\xi^i(t),\eta^i(t);f(t)\big)$
satisfies system \eqref{eq:NN} on $[0,T]$. In particular,
 for $d\mu_0$-a.a. $(x,v)\in S_0$ and $\forall t\in[0,T]$ we have $|x(t)-\xi^i(t)|>0$ for all $i$ and
 \begin{equation}
  f(x(t),v(t),t)=f_0(x,v).\label{invN}
\end{equation}
Finally,
\begin{equation}
  |\rho(x,t)|\leq C\big(1+\max_i\ln_-|x-\xi^i(t)|\big),\quad (x,t)\in \R^2\times [0,T]. \label{p1N}
\end{equation}
\end{thm}
To prove Theorem \ref{thm:mainN}
 we establish some a priori estimates assuming that a smooth solution to system \eqref{eq:NN} does exist. 
This will allow us to complete the proof by means of a regularization procedure as the one we have used in Section 
\ref{section:compactness} in the case of a single charge.

 The total energy of this system reads now
\begin{equation}
\label{EN}\begin{split}
& \mathcal{E}(t)
=\frac12\int |v|^2f(x,v,t)\,dx\,dv
+\sum_{i=1}^N\frac{|\eta^i(t)|^2}{2}
-\sum_{ \substack {i\neq j=1 }}^N\ln |\xi^i(t)-\xi^j(t)|\\ &-\frac12\iint \ln|x-y|\rho(x,t)\rho_(y,t)\,dx\,dy
+\sum_{i=1}^N\int \ln |x-\xi^i(t)|\rho(x,t)\ dx+C.\end{split}
\end{equation}
This is a conserved quantity along the solutions to system \eqref{eq:NN} and, as before 
(see Proposition \ref{prop:E0}), the assumption \eqref{max} 
on the support of $f_0$ ensures that  $\mathcal{E}(0)$ is bounded and positive 
for a choice of the constant $C$ large enough. However we cannot a priori exclude that, by compensation, a couple of 
charges moves closer while another couple, or the plasma, go to infinity. 
Actually our strategy for the proof of Theorem \ref{thm:mainN} consists first of all in establishing
 a lower bound for the mutual distances between the charges. 
This will imply that a plasma particle can approach at most one charge at the same time and 
consequently the $N$-charges case will be reduced to a sequence of one-charge problems, which we already solved.
In order to show that the $N$ point charges remain separated during the motion over the time interval 
$[0,T]$ we introduce the kinetic energy of the system
$$
K(t) =\frac12\int |v|^2f(x,v,t)\,dx\,dv+\sum_{i=1}^N\frac{|\eta^i(t)|^2}{2}
$$
and we prove the following result:
\begin{prop}
$$
\sup_{t\in[0,T]}K(t)\leq C_5.
$$
\label{kinN}
\end{prop}
\begin{proof}
The bound can be proven by paraphrasing Proposition \ref{prop:unif}. Indeed,
\begin{equation}
\begin{split}
K(t)&=  \mathcal{E}(0)+\sum_{i\neq j=1}^N\ln |\xi^i(t)-\xi^j(t)|+\frac12\iint \ln|x-y|\rho(x,t)\rho(y,t)\,dx\,dy
\\&- \sum_{i=1}^N\int \ln |x-\xi^i(t)|\rho(x,t)\,dx\\
&\leq\mathcal{E}(0)+\Big(\sum_{ i\neq j=1}^N \Big)^* \ln |\xi^i(t)-\xi^j(t)|
+\frac12\iint_{|x-y|\geq 1} \ln|x-y|\rho(x,t)\rho(y,t)\,dx\,dy\\
&+\sum_{i=1}^N\int \ln_- |x-\xi^i(t)|\rho(x,t)\,dx
\end{split}
\label{KN}\end{equation}
where $\Big(\sum \Big)^*$ stands for the sum over all $i\neq j$ such that $ |\xi_\eps^i(t)-\xi_\eps^j(t)|\geq 1.$ 
The last two terms in \eqref{KN} can be estimated as we did in \eqref{pp} 
in Proposition \ref{prop:unif}, so that by \eqref{ineq:rho3} we get: 
\begin{equation}
\begin{split}
\frac12\iint_{|x-y|\geq 1}& \ln|x-y|\rho(x,t)\rho(y,t)\,dx\,dy+\sum_{i=1}^N\int \ln_- |x-\xi^i(t)|\rho(x,t)\,dx\\
&\leq C\|\rho(t)\|_{L^2}+\frac12\iint (|x|+|y|)\rho(x,t)\rho(y,t)\,dx\,dy\\
&\leq  C\sqrt{K(t)}+\int |x|\rho_\eps(x,t)\,dx\\
&\leq C\sqrt{K(t)}+C\int_0^t\sqrt{K(s)}ds.
\end{split}
\label{kN0}\end{equation}
Besides, we have
\begin{equation}
\begin{split}
\Big(\sum_{ i\neq j=1}^N \Big)^*\ln |\xi^i(t)-\xi^j(t)|&\leq \sum_{ \substack {i, j=1 }}^N (|\xi^i(t)|+|\xi^j(t)|)\\
&\leq C+\sum_{ \substack {i,j=1 }}^N \int_0^tds\ (|\eta^i(s)|+|\eta^j(s)|)\\
&\leq C+\int_0^t\sqrt{K(s)}\,ds.
\end{split}
\label{kN} \end{equation}

Finally, combining \eqref{KN},  \eqref{kN0} and \eqref{kN} we obtain:
\begin{equation}\label{ineq:rho4}
K(t)\leq \mathcal{E}(0)+C\Big(1+\sqrt{K(t)}\Big)+\int_0^t\sqrt{K(s)}\,ds.
\end{equation}
 We conclude by means of Gronwall's Lemma.
\end{proof}

\medskip

As a consequence of Proposition \ref{kinN} we have the result we were looking for.
\begin{corollaire}\label{coro:dist}

There exists $d=d(T)>0$ such that:
\begin{equation}
\min_{i\neq j}\inf_{t\in [0,T]} |\xi^i(t)-\xi^j(t)|\geq d.
\label{dist1}
\end{equation}
\end{corollaire}
\begin{proof}
By definition \eqref{EN} of the energy we have:
 \begin{equation}
 \begin{split}
-\sum_{i\neq j=1}^N\ln |\xi^i(t)-\xi^j(t)|&= -C_4+\mathcal{E}(0)-K(t)\\
+ \frac12\iint \ln|x-y|&\rho(x,t)\rho(y,t)\,dx\,dy-\sum_{i=1}^N\int \ln |x-\xi^i(t)|\rho(x,t)\,dx\\
&\leq
\mathcal{E}(0)+\frac12\iint_{|x-y|>1} \ln|x-y|\rho(x,t)\rho(y,t)\,dx\,dy\\ &+\sum_{i=1}^N\int \ln_- |x-\xi^i(t)|\rho(x,t)\,dx.
\end{split}\label{poten}
\end{equation}
 But $|x-y|>1$ implies $0< \ln|x-y|<|x|+|y|,$ so that 
by \eqref{ineq:rho3} and Proposition \ref{kinN} the first integral in \eqref{poten}
 is bounded by a constant. For the second one we recall the 
bounds \eqref{rho^2}, \eqref{ineq:rho} and again Proposition \ref{kinN} to conclude:
 $$
 \int \ln_- |x-\xi^i(t)|\rho(x,t)\,dx\leq  C.
 $$
 
Hence we have proved that
\begin{equation}
-\sum_{i\neq j=1}^N\ln |\xi^i(t)-\xi^j(t)|\leq C.
\label{dist}
\end{equation}

Now notice that the bound on the kinetic energy implies 
 $$\sup_{t\in[0,T]} |\xi^i(t)|\leq C,\quad i=1,\ldots,N,$$  so that, 
putting as before $
\Big(\sum\Big)^*$ to indicate the sum over all $i\neq j$ for which $|\xi^i(t)-\xi^j(t)|> 1,$ 
by \eqref{dist} we have:
\begin{equation*}\begin{split}
\sum_{ i\neq j=1}^N\ln_- |\xi^i(t)-\xi^j(t)|&=
-\sum_{i\neq j=1}^N\ln |\xi^i(t)-\xi^j(t)|+\Big(\sum_{ i\neq j=1}^N\Big)^{*}\ln |\xi^i(t)-\xi^j(t)|\\
&\leq\Big(\sum_{ i\neq j=1}^N\Big)^{*}\ln |\xi^i(t)-\xi^j(t)|+ C\\
&\leq
2\sum_{i=1}^N|\xi^i(t)|+ C\leq C
\end{split}\end{equation*}
and this implies the thesis.
\end{proof}

\medskip
Corollary \ref{coro:dist} is the key to pass from one to $N$ charges and  to establish the  

\medskip

{\it{Proof of Theorem \ref{thm:mainN}}.}

We denote by $S_t$ the support of the density $f(t)$. Next,  we set :
\begin{equation*}
\begin{split}
&\Lambda_i(t)=B\left(\xi^i(t),\frac{d}{8}\right)\times \R^2,
\quad \Lambda^c(t) = \bigcap_{i=1}^{N}\Lambda_i^c(t)\end{split}\end{equation*}
where $d$ is the constant in \eqref{dist1}, and 
\begin{equation*}\begin{split}
&{\cal{H}}_i(t)=\sup_{s\in [0,t]}\sup_{(x,v)\in \Lambda_i(s)\cap S_s}|h^i(x,v,s)|+C\\
&{\cal{H}}^c(t)=\sup_{s\in [0,t]}\max_i\sup_{(x,v)\in \Lambda^c(s)\cap S_s}|h^i(x,v,s)|+C,
\end{split}\end{equation*}
with $C$ sufficiently large for further purposes. Finally, let 
\begin{equation*}
{\cal{H}}(t)=\max \left\{\max_i {\cal{H}}_i(t), \ {\cal{H}}^c(t)\right \}.
\end{equation*}
\begin{remark}
Since $d_0/4>d/8$, there exists $C_1>0$ such that  
$S_0(C_0,d_0/4)\subset  S_0(C_1,d/8). $ Hence \eqref{max} implies that ${\cal{H}}(0)$ is finite.
 \end{remark}

Clearly, we have
\begin{equation*}
 |v|^2\leq 2 \mathcal{H}(t)+C\sum_{i=1}^N \ln_- |x-\xi^i(t)|,\quad \forall t\in [0,T],\quad 
\forall (x,v)\in S_t;
\end{equation*}
In particular, one can readily transpose to the present case the
a priori bounds given, in the single-charge case, for the electric field in terms of the largest 
energy (Propositions \ref{coro:E}, \ref{adjoined} and 
\ref{prop:al-lip}).
Hence, as before (Theorem \ref{thm:main2}), our goal is to prove that ${\cal{H}}(T)$ is bounded in term of the initial 
datum. The conclusion of Theorem \ref{thm:mainN} will then follow by simply mimicking the arguments 
of Section \ref{section:compactness}; we omit the details here. 
The strategy in obtaining an a priori bound on $\mathcal{H}(T)$ is the following. We consider a small time interval $[0,\Delta^*]$ given by
\begin{equation}
\Delta^*=\frac{d}{16\sqrt{3\mathcal{H}({t^*})}} \label{time}
\end{equation}
$t^*<T$ being a suitably small fixed time (see \eqref{choose}).  
We will prove that ${\cal{H}}({t^*})$ is bounded, by making an 
iteration procedure in time over the intervals $[(k-1)\Delta^*, k\Delta^*], k=1,\ldots,n=([t^*]+1)/\Delta^*.$ 
Since the time $t^*$ is only function of constants, the subsequent 
prolongation of the estimate from $t^*$ to $T$ will be straightforward.

\medskip

Now we claim:\\
\\
 {\bf{Claim}}: During the time interval $[0,\Delta^*]$ 
any characteristic of the plasma starting from $S_0(C_0,d_0/4)$ 
may approach at most one charge. More precisely, if
there exists an index $i$ for which
 \begin{equation}
 x(s_0)\in  B\left(\xi^i(s_0),\frac{d}{8}\right) \quad \text{for some } s_0\in [0, \Delta^*],
 \label{approach}
 \end{equation}
 then
 \begin{equation}
 x(t)\in B\left(\xi^i(t),\frac{3d}{8}\right) \quad \forall t\in [0,\Delta^*].\label{far}
 \end{equation}
\medskip

Indeed, \eqref{far} clearly holds if $x(s)\in  B\left(\xi^i(s),d/8\right)$ on $[0,\Delta^\ast]$. Otherwise we have
 $|x(s_1)-\xi^i(s_1)|=d/8$ for some $s_1$. We denote by $(s_-,s_+)$ the maximal connected component of $s_1$ in 
$[0,\Delta^\ast]$ 
on which 
\begin{equation*}
 \frac{d}{16}<|x(s)-\xi^i(s)|<\frac{3d}{8},\quad s\in(s_-,s_+).
\end{equation*}
Let $s\in (s_-,s_+)$. If $|x(s)-\xi^i(s)|\geq d/8$ then $x(s)\in \Lambda^c(s)$, so that
\begin{equation*} 
 |v(s)|\leq \sqrt{2\mathcal{H}({t^\ast})}+C< \sqrt{3\mathcal{H}({t^\ast})}.
\end{equation*}
Otherwise, we have $x(s)\in \Lambda_i(s)$ but $|x(s)-\xi^i(s)|>d/16$ hence we  obtain again
\begin{equation*}
  |v(s)|\leq \sqrt{2|h^i(x(s),v(s),s)|}+C\sqrt{|\ln|x(s)-\xi^i(s)|}< \sqrt{3\mathcal{H}({t^\ast})}.
\end{equation*}
By choice of  $\Delta^\ast$ (see \eqref{time}) we thus obtain
\begin{equation*}
\big| |x(s)-\xi^i(s)|-|x(s_1)-\xi^i(s_1)|\big|< \sqrt{3\mathcal{H}({t^\ast})}\Delta^\ast= \frac{d}{16}.
\end{equation*}
We conclude that $(s_-,s_+)=[0,\Delta^\ast]$ and \eqref{far} follows.

\medskip

Now let us consider $t\in (0,\Delta^*]$, an index $i$ and $(\bar{x},\bar{v}) \in \Lambda_i(t)\cap S_t.$ 
 Then it has to be
$$
(\bar{x},\bar{v})=(x(x,v,t),v(x,v,t))
$$
for some $(x,v)\in \Lambda_i (0)\cup \Lambda^c (0)$ since we just proved that 
 a characteristic may visit at most one ball $B(\xi^i,d/8)$ on $[0,\Delta^\ast]$. 

For $s\in [0,t]$ we compute
\begin{equation*}
\begin{split}
\frac{d}{dt}&h^i(x(s),v(s),s)=
( v(s)-\eta^i(s))\cdot  \left(E(x(s),s)+E(\xi^i(s),s)\right)\\ &-( v(s)-\eta^i(s))\cdot \sum_{j\neq i=1}^
N\left(F^j(\xi^i(s),s)+F^j(x(s),s)\right).
\end{split}\end{equation*}
Notice that, as in the case of one charge, the singular contribution due to the nearest $i$-th charge to $x$ 
disappears. For $j\neq i$, it follows from \eqref{far} that
\begin{equation*}
|F^j(\xi^i(ts),s)|+|F^j(x(s),s)|\leq \frac{C}{d} 
\end{equation*}
 so that the external fields $F^j$ are smooth gradient fields on $[0,\Delta^*]$ and the results stated in 
the preceding sections hold with minor modifications.  Consequently we arrive at the analogue of estimate \eqref{H}:
\begin{equation*}
|h^i(x(t),v(t),t)|\leq {\mathcal{H}}(0)+C\sqrt{ \mathcal{H}(t)}+C\int_0^t\ {\mathcal{H}}(s) \ ds.
\end{equation*}
Taking the supremum we obtain
\begin{equation}
\sup_{(\bar{x},\bar{v})\in \Lambda_i(t)\cap S_t}|h^i(\bar{x},\bar{v},t)|\leq  {\mathcal{H}}(0)+
 C\sqrt{ \mathcal{H}(t)}+C\Delta^*{\mathcal{H}}({\Delta^*}).
 \label{acca}
\end{equation}

Consider now $(\bar{x},\bar{v})\in \Lambda^c(t)\cap S_t.$ There exists $(x,v)$ such that 
$$
(\bar{x},\bar{v})=(x(x,v,t),v(x,v,t))
$$
and  $(x,v)\in \Lambda _i(0)\cup \Lambda^c(0)$ for some $i.$ Proceeding as above we get again:
\begin{equation}
|h^i(\bar{x},\bar{v},t)|\leq {\mathcal{H}}(0)+ C\sqrt{ \mathcal{H}(t)}+C\Delta^*{\mathcal{H}}({\Delta^*}).
\label{acca'}
\end{equation}
Now we observe that since $(\bar{x},\bar{v})\in \Lambda^c(t),$ for any $j\neq i$ we have
\begin{equation*}
\begin{split}
|h^j (\bar{x},\bar{v},t)|&\leq |h^i (\bar{x},\bar{v},t)|+C|\bar{v}|+\left|\ln|\bar{x}-\xi^i(t)|-
\ln|\bar{x}-\xi^j(t)|\right|\\&
\leq |h^i (\bar{x},\bar{v},t)|+C\sqrt{ \mathcal{H}(t)}
\end{split}
\end{equation*}
so that by \eqref{acca'} we find
\begin{equation}
\max_{i=1,\ldots,N}\sup_{(\bar{x},\bar{v})\in \Lambda^c(t)\cap S_t}|h^i(\bar{x},\bar{v},t)|
\leq {\mathcal{H}}(0)+ C\sqrt{ \mathcal{H}(t)}+C\Delta^*{\mathcal{H}}({\Delta^*}). \label{acca''}
\end{equation}
Hence  \eqref{acca}, \eqref{acca''} and the monotonicity of $\mathcal{H}(t)$ in $t$ imply for any $t\in [0, \Delta^*]$:
\begin{equation*}
{\mathcal{H}}({\Delta^*})\leq {\mathcal{H}}(0)+ C\sqrt{ \mathcal{H}({\Delta^*})}
+C\Delta^*{\mathcal{H}}({\Delta^*})\leq {\mathcal{H}}(0)+ C\sqrt{ \mathcal{H}({t^*})}+C\Delta^*{\mathcal{H}}({t^*})
\end{equation*}
provided the constant in the definition of $\mathcal{H}(t)$ is large enough. 
Recalling the definition \eqref{time} of $\Delta^*$  we arrive at
\begin{equation*}
{\mathcal{H}}({\Delta^*})\leq {\mathcal{H}}(0)+ C\Delta^*{\mathcal{H}}({t^*}).
\end{equation*}

By iterating the procedure $n=([t^*]+1)/\Delta^*$ times we get:
$$
{\mathcal{H}}({t^*})\leq {\mathcal{H}}({ n\Delta^*})
\leq {\mathcal{H}}(0)+Cn\Delta^*{\mathcal{H}}({t^*})={\mathcal{H}}(0)+Ct^*{\mathcal{H}}(t^*)
$$
which implies
\begin{equation*}
{\mathcal{H}}(t^*)\leq {\mathcal{H}}(0)\left(\frac{1}{1-Ct^*}\right).
\end{equation*}
Hence by fixing
\begin{equation*}
\label{choose}t^*=1/2C
\end{equation*} we get a bound on ${\mathcal{H}}({t^*}):$
\begin{equation}
{\mathcal{H}}(t^*)\leq 2{\mathcal{H}}(0). \label{double}
\end{equation}
 To go from $t^*$ to $2t^*$ we iterate the procedure, by dividing the interval $[t^*,2t^*]$ into 
$n$ intervals of width $\Delta^*=d/(16\sqrt{3\mathcal{H}({2t^*})})$ with $n=([t^*]+1)/\Delta^*$ and we find
$$
{\mathcal{H}}({2t^*})\leq {\mathcal{H}}({t^*})+Ct^*{\mathcal{H}}({2t^*})
$$
from which by \eqref{double} it follows
\begin{equation*}
{\mathcal{H}}({2t^*})\leq 2{\mathcal{H}}(0)\left(\frac{1}{1-Ct^*}\right)=4{\mathcal{H}}(0).
\end{equation*}
By the choice of $t^*$ the iteration stops at $Nt^*=T$ so that for $N=2C([T]+1)$ we finally obtain 
\begin{equation}
{\mathcal{H}}({T})\leq 2^N{\mathcal{H}}(0)
\leq C^T{\mathcal{H}}(0). \label{final}
\end{equation}
The conclusion of Theorem \ref{thm:mainN} now follows from \eqref{final} and a suitable adaptation of the arguments
presented in Section \ref{section:compactness}, as was explained at the beginning of this section. \qed

\end{document}